\documentclass[10pt]{amsart}
\usepackage{amsmath}
\usepackage{amsthm}
\usepackage{amsopn}
\usepackage{amssymb}
\usepackage[all]{xy}
\usepackage{lscape,xcolor}
\usepackage{graphicx}
\usepackage{hyperref}
\usepackage{mathtools}
\usepackage{stmaryrd}
\usepackage{subcaption} 
\usepackage{afterpage}

\allowdisplaybreaks[1]

\newcommand{\mmod}{\! \sslash \!}
\newcommand{\E}[2]{\prescript{#1}{#2}{E}}

\newcommand{\mc}[1]{\mathcal{#1}}
\newcommand{\ul}[1]{\underline{#1}}
\newcommand{\mb}[1]{\mathbb{#1}}
\newcommand{\mr}[1]{\mathrm{#1}}

\newcommand{\abs}[1]{\lvert #1 \rvert}

\newcommand{\bra}[1]{\langle #1 \rangle}
\newcommand{\br}[1]{\overline{#1}}

\newcommand{\td}[1]{\widetilde{#1}}

\newcommand{\ZZ}{\mathbb{Z}}

\newcommand{\QQ}{\mathbb{Q}}

\newcommand{\FF}{\mathbb{F}}

\newcommand{\TMF}{\mathrm{TMF}}

\newcommand{\tmf}{\mathrm{tmf}}
\newcommand{\bo}{\mathrm{bo}}

\def \HF2{\mr{H}\FF_2}

\newcommand{\bou}{\ul{\bo}}
\newcommand{\tmfu}{\ul{\tmf}}

\newcommand{\bp}[1]{\bou_1^{\otimes {#1}}}

 \newtheorem{thm}[equation]{Theorem}
 \newtheorem{cor}[equation]{Corollary}
 \newtheorem{lem}[equation]{Lemma}
 \newtheorem{prop}[equation]{Proposition}
 
 \newtheorem{conj}[equation]{Conjecture}

 \newtheorem*{thm*}{Theorem}
 \newtheorem*{cor*}{Corollary}
 \newtheorem*{lem*}{Lemma}
 \newtheorem*{prop*}{Proposition}

 \theoremstyle{definition}

 \newtheorem{rmk}[equation]{Remark}

\newtheorem*{defn*}{Definition}
\newtheorem*{ex*}{Example}
\newtheorem*{exs*}{Examples}
\newtheorem*{rmk*}{Remark}
\newtheorem*{claim*}{Claim}

\numberwithin{equation}{section}
\numberwithin{figure}{section}
\DeclareMathOperator{\Ext}{Ext}
\DeclareMathOperator{\Hom}{Hom}

\DeclareMathOperator*{\hocolim}{hocolim}

\DeclareMathOperator{\sq}{Sq}

\newcommand{\s}{\wedge}

\newcommand{\T}{\ul{\TMF_0(3)}}
\newcommand{\piA}{\pi^{A(2)_*}}

\definecolor{DarkBlue}{rgb}{.1, 0.35, 0.6} 
\newcommand{\bl}[1]{#1}

\title{The structure of the $v_2$-local algebraic $\tmf$ resolution}

\author{M.~Behrens}\address{University of Notre Dame}\email{mbehren1@nd.edu}
\author{P.~Bhattacharya}\address{New Mexico State University}\email{prasit@nmsu.edu}
\author{D.~Culver}\address{Max Plank Institute for Mathematics }\email{dculver@mpim-bonn.mpg.de}

\begin{document}

\begin{abstract}
\bl{We give a complete description of the $E_1$-term of the $v_2$-local as well as $g$-local algebraic $\tmf$ resolution. }
\end{abstract}

\maketitle	

\tableofcontents

\section{Introduction}

Let $\bo$ denote the connective real $K$-theory spectrum.
Mahowald and his collaborators used the $\bo$ resolution (aka the $\bo$-based Adams spectral sequence) to study stable homotopy groups to great effect.  Specifically, they computed the image of the $J$-homomorphism \cite{DavisMahowaldJ}, proved the $2$-primary height $1$ telescope conjecture \cite{Mahowald}, \cite{LellmannMahowald}, computed the unstable $v_1$-periodic homotopy groups of spheres \cite{MahowaldEHP}, and applied homotopy theoretic methods to a variety of geometric problems \cite{DavisGitlerMahowald}.

The spectrum $\bo$ has two distinct advantages that lend itself to these applications at the prime $2$.  Firstly, $\pi_0\bo$ is torsion free and $\pi_*\bo$ is Bott periodic (i.e. $v_1$-torsion free), so it is equipped to detect the zeroth and first layers of the chromatic filtration.  Secondly, $v_1$-periodic homotopy at the prime 2 is more complicated than at odd primes, and this is witnessed by the elements $\eta$ and $\eta^2$ generating additional anomalous torsion \cite{Adams}.  These elements and their $v_1$-multiples are detected by the $\bo$-Hurewicz homomorphism
$$ \pi_*^s \to \pi_* \bo. $$

At chromatic height $2$, the $2$-primary stable stems have a vast collection of anomalous torsion, and a significant portion of this $v_2$-periodic torsion is detected by the spectrum $\tmf$ of topological modular forms (see \cite{tmfhi}).
As such the $\tmf$ resolution represents a significant upgrade to the $\bo$ resolution.  Indeed, partial analysis of the $\tmf$ resolution has resulted in numerous powerful results \cite{BHHM}, \cite{BHHM2}, \cite{tmfZ}, \cite{tmfhi}.

For a spectrum $X$, the \emph{tmf resolution} of $X$ is the tower of cofiber sequences
\begin{equation}\label{eq:tmfres}
\xymatrix{
X \ar[d] & \Sigma^{-1}\br{\tmf} \wedge X \ar[d] \ar[l] & \Sigma^{-2}\br{\tmf}^{\wedge 2} \wedge X \ar[d] \ar[l] & 
\cdots \ar[l]
\\
\tmf \wedge X & \Sigma^{-1}\tmf \wedge \br{\tmf} \wedge X & \Sigma^{-2}\tmf \wedge \br{\tmf}^{\wedge 2} \wedge X
}
\end{equation}   
Here $\br{\tmf}$ is the cofiber of the unit
$$ S \to \tmf \to \br{\tmf}. $$
Applying $\pi_*$ to the tower above results in the \emph{$\tmf$-based Adams spectral sequence}
$$ \E{\tmf}{}^{n,t}_1(X) = \pi_{t} (\tmf \wedge \br{\tmf}^{\wedge n} \wedge X) \Rightarrow \pi_{t-n} X.
$$ 

Ultimately, the successful applications of the $\tmf$-resolution so far have been limited by our ability to compute the $E_1$-page of the $\tmf$-based Adams spectral sequence --- computations to date have relied on computations of the $E_1$-page in certain regions.  Unlike the $\bo$ case, we are not able to completely compute this $E_1$ page for $X = S$.  The goal of this paper is to make a significant step towards rectifying this deficiency.

The computations of the $E_1$-page that have been successfully performed used the classical Adams spectral sequence.  We focus our attention at the prime $2$.  Recall that for a connective spectrum $Y$, the \emph{mod $2$ Adams spectral sequence} (ASS) takes the form
$$ \E{ass}{}^{s,t}_2(Y) = \Ext^{s,t}_{A_*}(\FF_2, H_*Y) \Rightarrow \pi_{t-s}Y^{\wedge}_2 $$
where $H_*$ denotes mod $2$ homology and $A_*$ is the dual Steenrod algebra.
The $E_1$-term of the $\tmf$-resolution can then itself be approached by computing the ASS's
$$ \E{ass}{}^{s,t}_2(\tmf \wedge \br{\tmf}^{n} \wedge X) \Rightarrow \pi_{t-s}(\tmf \wedge \br{\tmf}^n \wedge X) = \E{\tmf}{}_1^{n,t-s}(X). $$
In practice, given the computation of the $E_2$-pages, these Adams spectral sequences can be completely computed, as the majority of the differentials can be deduced from the Adams spectral sequence for $\tmf$ (as computed in \cite{BrunerRognes}).  The $\tmf$-resolution can then be studied through the Miller square \cite{Miller}
$$ 
\xymatrix{
	\E{ass}{}^{s,t}_2(\tmf \wedge \br{\tmf}^{n} \wedge X) \ar@{=>}[r]^-{ASS} \ar@{=>}[d]|{\text{alg tmf res}} & \E{\tmf}{}_1^{n,t-s}(X) \ar@{=>}[d]|{\text{tmf res}}
	\\
	\E{ass}{}^{s+n,t}_2(X) \ar@{=>}[r]_{ASS} & \pi_{t-s-n}X^{\wedge}_2
}
$$
Here, the left side of the square is the \emph{algebraic $\tmf$-resolution}, the analog of the $\tmf$-resolution obtained by applying $\Ext_{A_*}$ to (\ref{eq:tmfres}).  The starting point is therefore the computation of the $E_1$-page of the algebraic tmf resolution of the sphere
$$ \E{ass}{}^{s,t}_2(\tmf \wedge \br{\tmf}^{n}). $$

Analogous to the case of the $\bo$-resolution and the $BP\bra{2}$-resolution \cite{Mahowald} \cite{Culver}, we propose the following conjecture.
\begin{conj}
	The map 
	$$ \E{ass}{}^{s,t}_2(\tmf \wedge \br{\tmf}^{n}) \rightarrow v_2^{-1}\E{ass}{}^{s,t}_2(\tmf \wedge \br{\tmf}^{n}) $$
	is injective for $s > 0$. 
\end{conj}
This conjecture is consistent with computations in low degrees (see, for instance, \cite{BOSS}).  It implies a good-evil decomposition of the $\tmf$-resolution of the sphere, analogous to that of \cite{boASS}, \cite{tmfZ}.

In this paper we give a complete computation of 
$$ v_2^{-1}\E{ass}{}^{*,*}_2(\tmf \wedge \br{\tmf}^{n}). $$
We now summarize the main results.

For a graded Hopf algebra $\Gamma$ over $k$, let $\mc{D}_{\Gamma}$ denote Hovey's stable homotopy category of $\Gamma$-comodules.  Briefly, $\mc{D}_{\Gamma}$ is similar to the derived category, with the chief difference that weak equivalences are defined to be the $\pi^{\Gamma}_{*,*}$-isomorphisms, where for a $\Gamma$-comodule $M$, the homotopy groups $\pi^{\Gamma}_{*,*}$ are defined to be
$$ \pi^{\Gamma}_{n,s}(M) := \Ext^{s, s+n}_{\Gamma}(k, M). $$
For $M \in \mc{D}_{\Gamma}$, we let $\Sigma^{n,s}M$ denote a shift in internal degree by $s+n$ and in cohomological degree by $s$, so we have
$$ \pi^{\Gamma}_{k,l}(\Sigma^{n,s}M) = \pi^{\Gamma}_{k-n, l-s}(M) $$
and
$$ [\Sigma^{n,s}k, M]_{\Gamma} = \pi^{\Gamma}_{n,s}(M). $$
Note that with our conventions $\Sigma^n = \Sigma^{n,0}$, and exact triangles in $\mc{D}_\Gamma$ take the form
$$ A \to B \to C \to \Sigma^{1,-1} A. $$ 
For a spectrum $X$, we shall let
$$ \ul{X} \in \mc{D}_{A_*} $$
denote the object associated to the mod $2$ homology $H_*X$.
In this notation the ASS takes the form
$$ \E{ass}{}^{s,t}_2(X) = \pi^{A_*}_{t-s,s}(\ul{X}) \Rightarrow \pi_{t-s}X^{\wedge}_2. $$ 

Since $\ul{\tmf} = (A \mmod A(2))_*$ \cite{Mathew} (where $A(2)$ is the subalgebra of the mod $2$ Steenrod algebra generated by $\sq^1$, $\sq^2$, and $\sq^4$), we have a change of rings isomorphism
\begin{equation}\label{eq:COR}
 \pi^{A_*}_{*,*}(\ul{\tmf} \otimes M) \cong \pi^{A(2)_*}_{*,*}(M)
 \end{equation}
 for any $M \in \mc{D}_{A_*}$. 
Therefore the $E_1$-term of the algebraic $\tmf$-resolution takes the form
$$ \E{ass}{}_2^{*,*}(\tmf \wedge \br{\tmf}^{\wedge n}) \cong \pi_{*,*}^{A(2)_*} (\br{\ul{\tmf}}^{\otimes n}). $$

There is a decomposition \cite{BHHM}
\begin{equation}\label{eq:splittingtmf^n}
 \br{\ul{\tmf}}^{\otimes n} \simeq \bigoplus_{i_1, \ldots, i_n > 0} \Sigma^{8(i_1 + \cdots + i_n)} \bou_{i_1} \otimes \cdots \otimes \bou_{i_n}
\end{equation} 
in $\mc{D}_{A(2)_*}$, where $\bou_i$ denotes the homology of the $i$th bo-Brown-Gitler spectrum
(see Section~\ref{sec:BrownGitler}).

For an object $M \in \mc{D}_{A(2)_*}$, the localization $v_2^{-1}M$ denotes the localization of $M$ with respect to the element
$$ v_2^{8} \in \pi^{A(2)_*}_{48,8}(\FF_2), $$ 
so we have
$$ v_2^{-1}\E{ass}{}_2^{*,*}(\tmf \wedge \br{\tmf}^{\wedge n}) \cong \pi_{*,*}^{A(2)_*} (v_2^{-1}\br{\ul{\tmf}}^{\otimes n}). $$
We will prove
\begin{thm}[see Corollary \ref{cor:v2bo2j} and (\ref{eq:v2boj2})]\label{thm:bo2j}
There are equivalences in $\mc{D}_{A(2)_*}$
\begin{align*}
	v_2^{-1} \bou_{2j} & \simeq \Sigma^{8j} v_2^{-1} \bou_j \oplus \Sigma^{8j+8,1} v_2^{-1} \bou_{j-1}, \\
	v_2^{-1} \bou_{2j+1} & \simeq v_2^{-1} \Sigma^{8j}\bou_j \otimes \bou_1.
\end{align*}
\end{thm}

The splittings of (\ref{eq:splittingtmf^n}) and Theorem~\ref{thm:bo2j} inductively imply that in $\mc{D}_{A(2)_*}$ the objects $v_2^{-1}\br{\ul{\tmf}}^{\otimes n}$ split as a wedge of bigraded suspensions of $v_2^{-1}\bou_1^{\otimes k}$.  We are left with identifying these explicitly.

To this end we will introduce an object
$$ \ul{\TMF_0(3)} \in \mc{D}_{A(2)_*} $$
which serves as an algebraic version of the $\tmf$-module $\TMF_0(3)$ (the theory of topological modular forms associated to the congruence subgroup $\Gamma_0(3) < SL_2(\ZZ)$), and prove

\begin{thm}[Proposition \ref{prop:split1} and \ref{prop:split2}]\label{thm:mainsplittings} 
	There are splittings in $\mc{D}_{A(2)_*}$
	\begin{align*}
		v_2^{-1} \bou_1^{\otimes 3} & \simeq 2\Sigma^{16,1} v_2^{-1}\bou_1 \oplus \Sigma^{24,2}\ul{\TMF_0(3)}, \\
		\ul{\TMF_0(3)} \otimes \bou_1 & \simeq \Sigma^{24,3} \ul{\TMF_0(3)} \oplus \Sigma^{40,6} \ul{\TMF_0(3)}.
	\end{align*}
\end{thm}

The splittings of Theorem~\ref{thm:mainsplittings} imply that the objects $v_2^{-1} \bou_1^{\otimes k}$ split in $\mc{D}_{A(2)_*}$ as a direct sum of bigraded suspensions of copies of $v_2^{-1}\FF_2$, $v_2^{-1}\bou_1$, $v_2^{-1}\bou_1^{\otimes 2}$, and $\ul{\TMF_0(3)}$.

Putting this all together, we have the following theorem (see Corollary~\ref{cor:main} for a more precise formulation).

\begin{thm*}
There is a splitting of 
$$ v_2^{-1} \br{\ul{\tmf}}^{\otimes n} \in \mc{D}_{A(2)_*}$$
into a well-described sum of various bigraded suspensions of 
\begin{itemize}
\item $v_2^{-1}\FF_2$,
\item $v_2^{-1}\bou_1$, 
\item $v_2^{-1}\bou_1^{\otimes 2}$,
\item $\ul{\TMF_0(3)}$.
\end{itemize}
\end{thm*}

The most subtle step to all of this is the first equivalence of Theorem~\ref{thm:bo2j}.  Indeed an explicit exact sequence (see  (\ref{eq:boSES1}) of \cite{BHHM}) implies that $v_2^{-1}\bou_{2j}$ is built from $v_2^{-1}\Sigma^{8j}\bou_j$ and $v_2^{-1}\Sigma^{8j+8,1}\bou_{j-1}$ in $\mc{D}_{A(2)_*}$.  The hard part is showing that the attaching map between these two components is trivial.  This is accomplished by showing that if this attaching map is non-trivial, then it is non-trivial after $g$-localization
where $g$ is the generator of $\piA_{20,4}(\FF_2)$.  We then prove the $g$-local attaching map is trivial (see Corollary~\ref{cor:glocal} and Theorem~\ref{thm:glocal}), strengthening the results of \cite{BBT}.

\begin{thm*}
There is a splitting of 
$$ g^{-1} \br{\ul{\tmf}}^{\otimes n} \in \mc{D}_{A(2)_*}$$
into a well-described sum of various bigraded suspensions of 
\begin{itemize}
\item $g^{-1}\FF_2$,
\item $g^{-1}\bou_1$, 
\item $g^{-1}\bou_1^{\otimes 2}$.
\end{itemize}
\end{thm*}

The $v_2$-local results of this paper may be applied to understand the $\TMF$-resolution, where
$$ \TMF = \tmf[\Delta^{-1}]. $$
Namely, there are localized ASS's
$$ \pi^{A(2)_*}_{*,*}(v_2^{-1}\ul{\br{\tmf}}^{\otimes s} \otimes \ul{X})
\Rightarrow \pi_* (\TMF \wedge \br{\TMF}^{\wedge s} \wedge X)^{\wedge}_2. $$ 

Our $v_2$-local results also may be used to understand the $v_2$-localized algebraic tmf resolution
$$ v_2^{-1}\piA_{*,*}(\br{\tmfu}^{\otimes n} \otimes M) \Rightarrow v_2^{-1}\pi^{A_*}_{*,*}(M). $$
Here, the $v_2$-localized Ext groups $v_2^{-1}\pi^{A_*}_{*,*}$ are as defined in \cite{MahowaldShick}.

The $g$-local results of this paper may be applied to understand $g$-local Ext over the Steenrod algebra, using the $g$-local algebraic $\tmf$-resolution
$$ \piA_{*,*}(g^{-1} \br{\tmfu}^{\otimes n} \otimes M) \Rightarrow g^{-1}\pi^{A_*}_{*,*}(M). $$

\subsection*{Organization of the paper}

In Section~\ref{sec:BrownGitler} we reduce the study of $\tmfu$ to the bo-Brown-Gitler comodules $\bou_j$.  We review exact sequences which relate these comodules to $\bp{k}$.  Upon $v_2$-localization, we show that these exact sequences give complete decompositions of $v_2^{-1}\bo_j$ in terms of bigraded suspensions of $v_2^{-1}\bp{k}$ for various $k$, \emph{provided certain obstructions $\partial_{j'}$ vanish for $j' \le j/2$.}

In Section~\ref{sec:bo1^k} we review the structure of $\piA_{*,*}(\bp{k})$ for $0 \le k \le 4$.  These will form the computational input for the rest of the paper.

In Section~\ref{sec:TMF3} we construct $\T \in \mc{D}_{A(2)_*}$, our algebraic analog of $\TMF_0(3)$, and establish some basic properties.

In Section~\ref{sec:splitting} we prove a few key splitting theorems that inductively give complete decompositions of $\bp{k} \in \mc{D}_{A(2)_*}$ into indecomposable summands.  Provided the obstructions $\partial_{j'}$ vanish, we therefore get complete decompositions of $v_2^{-1}\bou_j$.

In Section~\ref{sec:genfuncs} we define certain generating functions which conveniently allow for algebraic computation of the putative decompositions of $v_2^{-1}\bou_j$.

In Section~\ref{sec:glocal} we explain the analogs of the $v_2$-local decompositions of $\bou_j$ and $\bp{k}$ in the $g$-local category.  The decompositions of $g^{-1}\bou_j$ depend on the vanishing of certain obstructions $\partial'_{j}$.

In Section~\ref{sec:partialj} we prove our main result: the obstructions $\partial_j$ and $\partial'_j$ vanish for all $j$.  This results in a complete decomposition of $v_2^{-1}\br{\tmfu}^{\otimes n}$ and $g^{-1}\br{\tmfu}^{\otimes n}$.

In Section~\ref{sec:BBT}, we relate our $g$-local results to the computations of Bhattacharya, Bobkova, and Thomas \cite{BBT}, providing a strengthening of their results.

Appendix~\ref{apx:charts} contains charts of $\pi^{A(2)_*}_{*,*}\bou_1^{\otimes k}$ for $0 \le k \le 4$ and $\pi^{A(2)_*}_{*,*}(\T)$.  These are referred to throughout the paper.

\bl{In Appendix~\ref{appendix}, we discuss a stable splitting of $\bo_1^{\s 3}$ and its relationship with Theorem~\ref{thm:mainsplittings}. 
}

\subsection*{Acknowledgments}

The authors are grateful for the comments and corrections of two referees.
The results of this paper were made possible with the assistance of the computational Ext software of R.~Bruner and A.~Perry, and the computer algebra systems Fermat and Sage.  The first author was supported by NSF grants
DMS-1547292 and DMS-2005476.

\section{$\bo$-Brown-Gitler comodules}\label{sec:BrownGitler}

In this section we reduce the analysis of $v_2^{-1}\br{\ul{\tmf}}^{\otimes n}$ to the analysis of $v_2$-local $\bo$-Brown-Gitler comodules. 
These are $A_*$-comodules which are the homology of the bo-Brown-Gitler spectra constructed by \cite{GoerssJonesMahowald}.
Mahowald used integral Brown-Gitler spectra to analyze the bo resolution \cite{Mahowald}.  The bo-Brown-Gitler comodules play a similar role in the algebraic tmf resolution \cite{BHHM}, \cite{MahowaldRezk}, \cite{DavisMahowald},  \cite{BOSS}, \cite{BHHM2}, \cite{tmfhi}.  

Endow the mod $2$ homology of $\bo$ 
$$ \bou \cong A\mmod A(1)_* = \FF_2[\zeta_1^4, \zeta_2^2, \zeta_3, \ldots] $$
(where $\zeta_i$ denotes the conjugate of $\xi_i \in A_*$) with a multiplicative grading by declaring the \emph{weight} of $\zeta_i$ to be 
\begin{equation}\label{eq:weight}
wt(\zeta_i) = 2^{i-1}. 
\end{equation}
The $i$th \emph{$\bo$-Brown-Gitler} comodule is the subcomodule
$$ \bou_i  \subset A \mmod A(1)_* $$
spanned by monomials of weight less than or equal to $4i$.

For an object $M \in \mc{D}_{A(2)_*}$, let 
$$ DM = \Hom_{\FF_2}(M,\FF_2) $$ 
be its $\FF_2$-linear dual.  We record the following useful result.

\begin{prop}\label{prop:Dbo1}
In $\mc{D}_{A(2)_*}$, there is an equivalence 
$$ v_2^{-1} D\bou_1 \simeq \Sigma^{-16,-1}v_2^{-1}\bou_1. $$
\end{prop}

\begin{proof}
Consider the short exact sequence
$$ 0 \to \bou_1 \to A(2)\mmod A(1)_* \to \Sigma^{17} D\bou_1 \to 0. $$
Since we have
$$ v_2^{-1}\pi^{A(2)_*}_{*,*} A(2) \mmod A(1)_* \cong v_2^{-1}\Ext_{A(1)_*}(\FF_2,\FF_2) = 0 $$
it follows that the connecting homomorphism in $\mc{D}_{A(2)_*}$
$$ \Sigma^{17} v_2^{-1}D\bou_1 \to \Sigma^{1,-1}v_2^{-1}\bou_1 $$
is an equivalence.
\end{proof}

Our interest in the bo-Brown-Gitler comodules stems from the fact that there is a splitting of $A(2)_*$-comodules \cite[Cor.~5.5]{BHHM}:
\begin{equation}\label{eq:tmfsplitting}
 \ul{\tmf} \cong \bigoplus_{i \ge 0} \Sigma^{8i} \bou_i 
 \end{equation}
 where $\Sigma^{8j}\bou_j$ is spanned by the monomials of 
$$ \tmfu = A\mmod A(2)_* = \FF_2[\zeta^{8}_1, \zeta_2^4, \zeta_3^2, \zeta_4, \ldots] $$ 
of weight $8j$.
We therefore have a splitting of $A(2)_*$-comodules
\begin{equation}\label{eq:E1decomp}
\br{\ul{\tmf}}^{\otimes n} \cong \bigoplus_{i_1, \ldots, i_n > 0}\Sigma^{8(i_1+\cdots+i_n)}\bou_{i_1} \otimes \cdots \otimes \bou_{i_n}.
\end{equation}

The object
$$ \Sigma^{8(i_1+\cdots+i_n)}\bou_{i_1} \otimes \cdots \otimes \bou_{i_n} \in \mc{D}_{A(2)_*}  $$
can be inductively built from $\bou_1^{\otimes k}$ by means of a set of exact sequences of $A(2)_*$-comodules which relate the $\bou_i$'s \cite[Sec.~7]{BHHM}:
\begin{gather}
0\to \Sigma^{8j} \ul{\bo}_j \to \ul{\bo}_{2j}\to A(2)\mmod A(1)_* \otimes \ul{\tmf}_{j-1}  \to \Sigma^{8j+9} \ul{\bo}_{j-1} \to 0 \label{eq:boSES1},
\\
0 \to \Sigma^{8j} \ul{\bo}_j \otimes \ul{\bo}_1 \to \ul{\bo}_{2j+1}\to A(2)\mmod A(1)_* \otimes \ul{\tmf}_{j-1} \to 0. \label{eq:boSES2}
\end{gather}
Here, $\ul{\tmf}_j$ is the $j$th $\tmf$-Brown-Gitler comodule --- it is the subcomodule of $\tmfu$
spanned by monomials of weight less than or equal to $8j$.

\begin{rmk}
Technically speaking, as is addressed in \cite[Sec.~7]{BHHM}, the comodules 
$$ A(2)\mmod A(1)_* \otimes \ul{\tmf}_{j-1}$$ 
in the above exact sequences have to be given a slightly different $A(2)_*$-comodule structure from the standard one arising from the tensor product.  However, this different comodule structure ends up being $\Ext$-isomorphic to the standard one.  As the analysis of this paper only requires 
\begin{align*}
 v_2^{-1}A(2)\mmod A(1)_* \otimes \ul{\tmf}_{j-1} & \simeq 0,   \\
 g^{-1}A(2)\mmod A(1)_* \otimes \ul{\tmf}_{j-1} & \simeq 0,   
\end{align*}
and these equivalences hold for the non-standard comodule structures, the reader can safely ignore this subtlety.
\end{rmk}

Since 
$$ v_2^{-1}A(2)\mmod A(1)_* \otimes \ul{\tmf}_{j-1} \simeq 0, $$  
The exact sequences (\ref{eq:boSES1}) and (\ref{eq:boSES2}) give rise to a cofiber sequence in $\mc{D}_{A(2)_*}$
\begin{equation}\label{eq:v2boj1}
\Sigma^{8j} v_2^{-1}\ul{\bo}_j \to v_2^{-1}\ul{\bo}_{2j}\to \Sigma^{8j+8,1} v_2^{-1}\ul{\bo}_{j-1} 
\end{equation}
and an equivalence
\begin{equation}\label{eq:v2boj2}
  \Sigma^{8j} v_2^{-1}\ul{\bo}_j \otimes \ul{\bo}_1 \simeq v_2^{-1}\ul{\bo}_{2j+1}. 
 \end{equation}
Thus, (\ref{eq:v2boj1}) and (\ref{eq:v2boj2}) inductively build 
$$ v_2^{-1}\bou_i \in \mc{D}_{A(2)_*} $$ 
out of $v_2^{-1}\bou_1^{\otimes k}$.

The connecting homomorphism of the cofiber sequence (\ref{eq:v2boj1}) 
\begin{equation}\label{eq:partialj}
 \partial_j : v_2^{-1}\Sigma^{8j+8,1}\bou_{j-1} \to v_2^{-1} \Sigma^{8j+1,-1} \bo_{j}
\end{equation}
is the obstruction to the cofiber sequence being split.  We will prove in Section~\ref{sec:partialj} that the connecting homomorphism $\partial_j = 0$ for all $j$, so we have
\begin{equation}\label{eq:splittingconj}
 v_2^{-1}\ul{\bo}_{2j} \simeq v_2^{-1}\Sigma^{8j} \ul{\bo}_j \oplus v_2^{-1}\Sigma^{8j+8,1} \ul{\bo}_{j-1}.
 \end{equation}


%
%
%
%

\section{The groups $\pi^{A(2)_*}_{*,*}(\bou_1^k)$
}\label{sec:bo1^k}

In the previous section we related the comodules $\bou_j$ to the comodules $\bou_1^{\otimes k}$.  We now review the structure of 
$$ \pi^{A(2)_*}_{*,*} \bou_1^{\otimes k} $$
for $0 \le k \le 4$.  For $k = 0$, this computation was initially performed by May \cite{MayA2} but was first published in \cite{ShimadaIwai}.  For $k = 1$, the computation appears in \cite{DavisMahowaldA2}.  For $k = 2,3$ these computations appeared in \cite[Sec.~6]{BHHM}, where a methodology for performing these computations for $k \ge 2$ is explained.  This same methodology was extended by the authors of this paper to perform the computation for $k = 4$.  It should be emphasized that use of the Ext software of Bruner \cite{Bruner} and Perry \cite{Perry} was crucial for the cases of $2 \le k \le 4$.

In order to give names to the $v_0$-torsion-free generators of $\piA_{*,*}(\bp{k})$, we review the corresponding $v_0$-local computations.  The entire structure of the $v_0$-local algebraic tmf resolution is given in \cite{tmfhi} (see also \cite{BOSS}). 

Observe that we have
\begin{equation}\label{v0localExtA2}
v_0^{-1}\pi^{A(2)_*}_{*,*}(\FF_2) = \FF_2[v_0^{\pm}, v_1^4, v_2^2].
\end{equation}
Note that $c_4, c_6 \in (\tmf_*)_\QQ$ are detected in the $v_0$-localized ASS by $v_1^4$ and $v_0^3v_2^2$, respectively.

We have (regarding $\bou_1$ as a subcomodule of $A\mmod A(2)_*$)
\begin{equation*}
v_0^{-1} \pi_{*,*}^{A(2)_*}(\bou_1) = \FF_2[v_0^{\pm}, v_1^4, v_2^2]\{\zeta_1^8, \zeta^4_2\}
\end{equation*}

We therefore have
an isomorphism 
\begin{equation}\label{eq:v0localbo1^k}
 v_0^{-1}\pi^{A(2)_*}_{*,*}(\bp{k}) \cong \FF_2[v_0^\pm, v_1^4, v_2^2]\otimes \FF_2\{\zeta^8_1, \zeta_2^4\}^{\otimes k}.
 \end{equation}
To make for more compact notation, we will use bars to denote elements of tensor powers:
\begin{equation}\label{eq:bar}
 x_1|\cdots |x_n := x_1 \otimes \cdots \otimes x_n.
 \end{equation}

{\bf $\pmb{\pi_{*,*}^{A(2)_*}(\FF_2):}$ (Figure \ref{fig:ExtA2})}

All of the elements are
$c_4 = v_1^4$-periodic, and $v_2^8$-periodic.  
Exactly one $v_1^4$ multiple
of each element is displayed with the $\bullet$ replaced by a $\circ$. 
Observe the wedge pattern beginning in $t-s = 35$.
This pattern is infinite, propagated horizontally by $h_{2,1}$-multiplication 
and vertically by $v_1$-multiplication.  Here, $h_{2,1}$ is the name of the
generator in the May spectral sequence of bidegree $(t-s,s) = (5,1)$, and
$h_{2,1}^4 = g$.

{\bf $\pmb{\pi_{*,*}^{A(2)_*}(\bou_1^{\otimes k}),\text{ for }k = 1,2,3,4:}$ (Figures \ref{fig:bo1}, \ref{fig:bo1^2}, \ref{fig:bo1^3}, \ref{fig:bo1^4})}

Every element is $v_2^8$-periodic.
However, unlike $\pi_{*,*}^{A(2)_*}(\FF_2)$, not every element of these Ext groups is
$v_1^4$-periodic.  Rather, it is the case that either an element $x \in
\Ext_{A(2)_*}(\bou_1^{\otimes k})$ satisfies 
$v_1^4x = 0$, or it is $v_1^4$-periodic.
Each of the $v_1^4$-periodic 
elements fit into families which look like shifted and truncated copies of
$\pi^{A(1)_*}_{*,*}(\FF_2)$, and are labeled with a $\circ$.  
We have only
included the beginning of these $v_1^4$-periodic patterns in the chart.  
The other
generators are labeled with a $\bullet$.  A $\Box$ indicates a
polynomial algebra $\FF_2[h_{2,1}]$.  Elements which are $v_0$-torsion-free are named in these charts using (\ref{eq:v0localbo1^k}), in the bar notation of (\ref{eq:bar}).

\section{An algebraic model of\text{ $\TMF_0(3)$}}\label{sec:TMF3}




\bl{The spectrum $\TMF_0(3)$ is an analog of $\TMF$ associated to the moduli of elliptic curves with with $\Gamma_0(3)$-structures introduced and studied by Mahowald and Rezk \cite{MahowaldRezk}. In fact, Mahowald and Rezk proposed three different connective spectra  whose $E(2)$-localizations are $\TMF_0(3)$ (also see \cite{DavisMahowald}).}

 \bl{We will emulate \cite{MahowaldRezk,  DavisMahowald} in the category of $\mc{D}_{A(2)_*}$ to construct the $\T$. }

\begin{lem}\label{lem:h22tilde}
The composite
$$ \Sigma^{6,2}\FF_2 \xrightarrow{h_2^2} \FF_2 \hookrightarrow \Sigma^7 D\bou_1 $$
extends to a map
$$ \td{h_2^2}: \Sigma^{6,2}\bou_1 \rightarrow \Sigma^7 D\bou_1. $$
\end{lem}

\begin{proof}
The cell structure of $\bou_1$ implies that the obstructions to this extension is the product $h_2 \cdot h_2^2 e_7$, and the Massey products $\bra{h_1, h_2, h_2^2 e_7}$ and $\bra{h_0, h_1, h_2, h_2^2 e_7}$.  These are all zero for dimensional reasons.
\end{proof}

Our algebraic model of $\TMF_0(3)$ is defined to be 
$$ \ul{\TMF_0(3)} := v_2^{-1}(\Sigma^{24,3}D\bou_1 \cup_{\td{h_2^2}}
\Sigma^{24,4}\bou_1) $$
where $\Sigma^{24,3}D\bou_1 \cup_{\td{h_2^2}}
\Sigma^{24,4}\bou_1$ denotes (the $\Sigma^{17,3}$ suspension of) the cofiber of the map $\td{h_2^2}$ of Lemma~\ref{lem:h22tilde}.

Figure~\ref{fig:Dbo1Ubo1} shows a computation of the homotopy of $D\bou_1 \cup_{\td{h_2^2}}
\Sigma^{0,1}\bou_1$.  In this figure, the solid dots correspond to $D\bou_1$ and the open dots correspond to $\bou_1$.  One convenient way of accessing the homotopy of $D\bou_1$ is from the short exact sequence in the proof of Proposition~\ref{prop:Dbo1}.

A chart of $\piA_{*,*}(\T)$ is displayed in Figure~\ref{fig:TMF3}.
Just like in the charts of Figures~\ref{fig:bo1}, \ref{fig:bo1^2}, \ref{fig:bo1^3}, \ref{fig:bo1^4}, each of the $v_1^4$-periodic 
elements fit into families which look like shifted and truncated copies of
$\pi^{A(1)_*}_{*,*}(\FF_2)$, and are labeled with a $\circ$, and only
beginning of these $v_1^4$-periodic patterns are included in the chart.  
The other
generators are labeled with a $\bullet$.  Figure~\ref{fig:TMF3} actually only displays the homotopy groups of a connective version of $\T$ (this is the object $\ul{X}$ from the proof of Proposition~\ref{prop:ring}).  However the $v_2$-periodic homotopy groups are easily deduced from the fact that these homotopy groups are all $v_2^{8}$-periodic.  Diagram~3.4 of \cite{DavisMahowald} gives a nice visualization of what these localized homotopy groups look like.

\begin{figure}
\includegraphics[angle = 90, origin=c, height =.6\textheight]{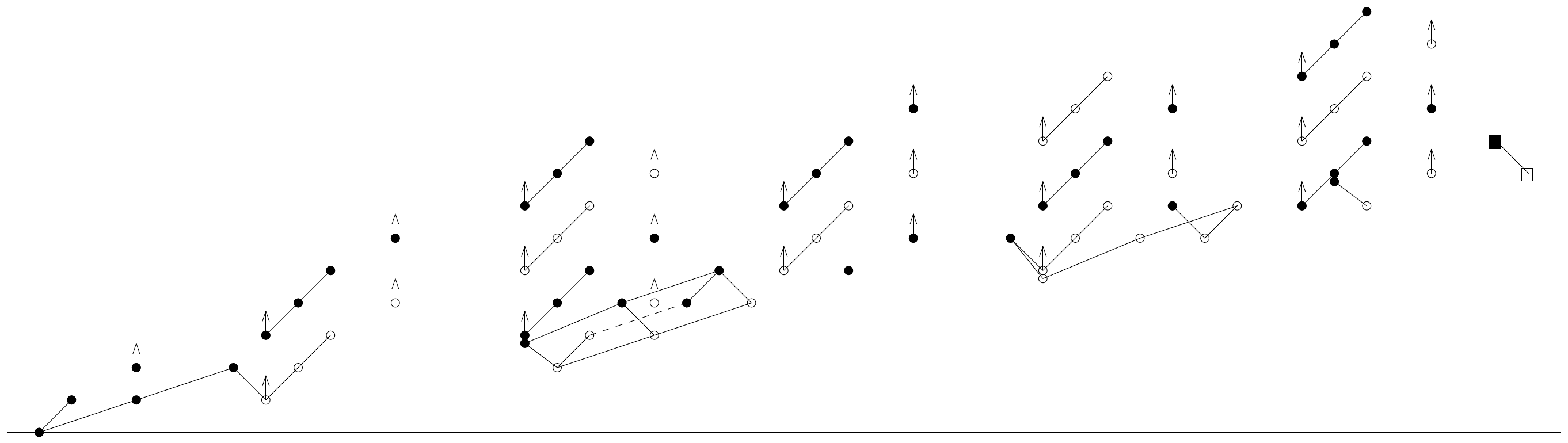}
\caption{Computing the homotopy of $D\bou_1 \cup_{\td{h_2^2}}
\Sigma^{0,1}\bou_1$.}\label{fig:Dbo1Ubo1}
\end{figure}

\begin{lem}\label{lem:rigidity}
Any map
$$ f: \ul{\TMF_0(3)} \to \ul{\TMF_0(3)} $$
which is the identity on $\pi^{A(2)_*}_{0,0}$ is an equivalence. 
\end{lem}

\begin{proof}
Let $1_{\T} \in \piA_{0,0}(\T)$ denote the generator.
The $\piA_{*,*}(\FF_2)$-module structure implies $f$ is the identity on 
$g\cdot 1_{\T}$ and $v_2^4h_1$.
It follows from $h_2$ linearity that $f$ is the identity on $x_{17}$ \bl{(see Figure~\ref{fig:TMF3})}.  Therefore $f$ is the identity on $v_2^4 h_1 x_{17}$.
It follows from $h_0$, $h_1$, $h_2$, and $v_1^4$ linearity that $f$ is an isomorphism on   $v_0^{-1}\piA_{*,*}(\T)$.  Here we must use the fact that the $v_0$-localization of $f$ is a map of $v_0^{-1}\pi_{*,*}(\FF_2)$-modules.  It then follows that $f$ is a $\piA_{*,*}$-isomorphism. 
\end{proof}

We have the following algebraic version of the Recognition Principle of Davis-Mahowald-Rezk (see \cite[Prop.~7.2]{MahowaldRezk}).

\begin{thm}[Recognition Principle]\label{thm:recprinc}
Suppose that $X \in \mc{D}_{A(2)_*}$ satisfies
\begin{equation}\label{eq:recprinciso}
 \pi^{A(2)_*}_{*,*}(X) \cong \pi^{A(2)_*}_{*,*}(\ul{\TMF_0(3)})
 \end{equation}
where the above isomorphism
preserves $v_0$, $h_1$, $h_2$, $v_1^4$, $v_0v_2^2$, $v_2^8$, $v_2^4h_1$, and $g$ multiplications.
Then there is an equivalence
$$ X \simeq \ul{\TMF_0(3)}. $$
\end{thm}

\begin{proof}
Let
$$ x_{17} : \Sigma^{17,3} \FF_2 \rightarrow X $$
represent the generator of $\pi^{A(2)_*}_{17,3}(X)$.  Since 
$$ \pi^{A(2)_*}_{17,4}(X) = \pi^{A(2)_*}_{19,4}(X) = \pi^{A(2)_*}_{23,4}(X) = 0, $$
there exists an extension of $x_{17}$ to a map
$$ \Sigma^{24,3}D\bou_1 \to X. $$ 
Since
$$ \pi^{A(2)_*}_{23,5}(X) = \pi^{A(2)_*}_{27,5}(X) = \pi^{A(2)_*}_{29,5}(X) = \pi^{A(2)_*}_{30,5}(X) = 0 $$
there exists a further extension of this map to a map
$$ \Sigma^{24,3}D\bou_1 \cup \Sigma^{24,4}\bou_1 \to X. $$
The conditions on the isomorphism (\ref{eq:recprinciso}) imply that $X \simeq v_2^{-1}X$.  Thus the map above localizes to a map
$$ v_2^{-1}(\Sigma^{24,3}D\bou_1 \cup \Sigma^{24,4}\bou_1) \to X. $$
The conditions on the isomorphism (\ref{eq:recprinciso}) then force the map above to be a $\pi^{A(2)_*}_{*,*}$-isomorphism.
\end{proof}

For us, \emph{a weak ring object} in $\mc{D}_{A(2)_*}$ is an object $R \in \mc{D}_{A(2)_*}$ with a unit
$$ u: \FF_2 \to R $$ 
and a multiplication
$$ m: R \otimes R \to R $$
such that the two composites
\begin{gather*}
R \otimes \FF_2 \xrightarrow{1 \otimes u} R \otimes R \xrightarrow{m} R,
\\
\FF_2 \otimes R \xrightarrow{u \otimes 1} R \otimes R \xrightarrow{m} R
\end{gather*}
are equivalences.

\begin{prop}\label{prop:ring}
$\ul{\TMF_0(3)}$ is a weak ring object in $\mc{D}_{A(2)_*}$.
\end{prop}  

\begin{proof}
We shall need to imitate the ``first model'' of \cite{MahowaldRezk}, \cite{DavisMahowald}.
Start with the $A_*$-comodule $\ul{Y}$ described in \cite[Thm.~2.1(a)]{DavisMahowald}.
Then the method of proof for \cite[Thm.~2.1(b)]{DavisMahowald} shows that there exists a map
$$ \td{h_0h_2} : \Sigma^{3,2}\ul{Y} \to \FF_2 $$
in $\mc{D}_{A_*}$ extending $h_0h_2$, so we can take the cofiber
$$ \ul{X} := \FF_2 \cup_{\td{h_0h_2}} \Sigma^{4,1} \ul{Y}. $$
Regarding this cofiber as an object of $\mc{D}_{A(2)_*}$, define
$$ R := v_2^{-1}\ul{X} \in \mc{D}_{A(2)_*}. $$
We will show (a) $R \simeq \ul{\TMF_0(3)}$ and (b) $R$ is a ring object of $\mc{D}_{A(2)_*}$.

For (a), we will compute $\pi^{A(2)_*}_{*,*}(R)$.  To this end, we observe that the methods of the proof of \cite[Thm.~2.1(c)]{DavisMahowald} show that there is a map
$$ f: \ul{X} \to A(2)\mmod A(1)_* $$
which extends the inclusion $\FF_2 \hookrightarrow A(2)\mmod A(1)_*$.  Let $\ul{C}$ be the cofiber of $f$:
\begin{equation}\label{eq:cofiberf}
\ul{X} \xrightarrow{f} A(2)\mmod A(1)_* 
\rightarrow \ul{C}.
\end{equation}
Then the proof of \cite[Thm.~2.1(d)]{DavisMahowald} shows that
$$ \pi^{A(2)_*}_{*,s}(A(2)_* \otimes \ul{C}) \cong 
\begin{cases}
\Sigma^4 A(2)/A(2)(\sq^4, \sq^5\sq^1 )_*, & s = 0, \\
0, & s > 0.
\end{cases}
$$
as an $A(2)_*$-comodule.  
The $A(2)_*$-based Adams spectral sequence for $\ul{C}$ then collapses to give an isomorphism
$$ \pi^{A(2)_*}_{n,s}(\ul{C}) \cong \Ext_{A(2)_*}^{s+n,s}(\FF_2, \Sigma^4 A(2)/A(2)(\sq^4, \sq^5\sq^1 )_*). $$
These Ext groups were computed in \cite[Thm.~2.9]{DavisMahowald}.
The cofiber sequence (\ref{eq:cofiberf}) gives an equivalence
$$ R \simeq \Sigma^{-1,1}v_2^{-1}\ul{C}. $$
We see by inspection of Davis-Mahowald's Ext computation alluded to above that there is an isomorphism
$$ \pi^{A(2)}_{*,*}(\Sigma^{-1,1}v_2^{-1}\ul{C}) \cong \pi^{A(2)_*}_{*,*}(\ul{\TMF_0(3)}) $$
satisfying the hypotheses of the Recognition Principle (Theorem~\ref{thm:recprinc}).  We deduce that there is an equivalence
$$ \ul{\TMF_0(3)} \simeq R. $$

We now just need to prove $R$ is a ring object in $\mc{D}_{A(2)_*}$.  For this we imitate the proof of \cite[Thm.~2.1(e)]{DavisMahowald}.  Namely, consider the composite
$$ \br{m}: \ul{X} \otimes \ul{X} \xrightarrow{f \otimes f} A(2)\mmod A(1)_* \otimes A(2) \mmod A(1)_* \xrightarrow{\mu} A(2)\mmod A(1)_*. $$
By the cofiber sequence (\ref{eq:cofiberf}), the map $\br{m}$ lifts to a map
$$ m: \ul{X} \otimes \ul{X} \to \ul{X} $$
if the composite
$$ \ul{X} \otimes \ul{X} \xrightarrow{\br{m}} A(2)\mmod A(1)_* \to \ul{C} $$
is null.  In the proof of \cite[Thm.~2.1(e)]{DavisMahowald}, it is established using Bruner's Ext software that 
$$ [\ul{X} \otimes \ul{X}, \ul{C}]_{A(2)_*} = 0. $$
Therefore, the lift $m$ exists.  Since it is a lift of $\br{m}$, it is the identity on the bottom cell.  It follows that the composites
\begin{gather*}
\ul{X} \otimes \FF_2 \hookrightarrow \ul{X} \otimes \ul{X} \xrightarrow{m} \ul{X},
\\
\FF_2 \otimes \ul{X} \hookrightarrow \ul{X} \otimes \ul{X} \xrightarrow{m} \ul{X}
\end{gather*}
are the identity on the bottom cell.  It follows from Lemma~\ref{lem:rigidity} that after $v_2$-localization, the composites 
\begin{gather*}
R \otimes \FF_2 \hookrightarrow R \otimes R \xrightarrow{m} R,
\\
\FF_2 \otimes R \hookrightarrow R \otimes R \xrightarrow{m} R
\end{gather*}
 are equivalences.  Thus $m$ gives $R$ the structure of a weak ring object.  (In fact, the analog of Lemma~\ref{lem:rigidity} holds for $\ul{X}$, and so $\ul{X}$ is also a weak ring object.)
\end{proof}

\section{Splitting $\bou_1^{\otimes k}$
}\label{sec:splitting}

In this section we prove our main $v_2$-local splitting theorems, which will be the basis of all of our subsequent $v_2$-local decomposition results.

\begin{prop}\label{prop:split1}
There is a splitting
$$ v_2^{-1} \bp{3} \simeq 2\Sigma^{16,1}v_2^{-1}\bou_1 \oplus \Sigma^{24,2} \T. $$
\end{prop}

\begin{proof}
Since we are working in characteristic 2, there is a decomposition
$$ \bp{3} \simeq (\bp{3})^{hC_3} \oplus B $$
where $C_3$ acts by cyclically permuting the terms, and we have 
$$ \pi^{A(2)_*}_{*,*}((\bp{3})^{hC_3}) = \pi^{A(2)_*}_{*,*}(\bp{3})^{C_3}. $$
It is easily checked, using the names of the generators in Figure~\ref{fig:bo1^3}, that there is an isomorphism
$$ v_2^{-1}\pi^{A(2)_*}_{*,*}((\bp{3})^{hC_3}) \cong \piA_{*,*}(\T). $$
A direct application of the Recognition Principle (Theorem~\ref{thm:recprinc}) shows that
$$ v_2^{-1}(\bp{3})^{hC_3} \simeq \Sigma^{24,2} \T. $$
Let
$$ x_{16}: \Sigma^{16,1} \FF_2 \to \bou_1^{\otimes 2} $$
correspond to the generator of $\pi_{16,1}^{A(2)_*}(\bp{2})$.  Then the composite
$$ \Sigma^{16,1} v_2^{-1}\bou_1 \oplus \Sigma^{16,1} v_2^{-1}\bou_1 \xrightarrow{x_{16} \otimes 1 \oplus 1 \otimes x_{16}} v_2^{-1}\bp{3} \to v_2^{-1}B $$
is seen to be a $\pi^{A(2)_*}_{*,*}$-isomorphism, hence an equivalence.
\end{proof}

\begin{prop}\label{prop:split2}
There is a splitting
$$ \T \otimes \bou_1 \simeq \Sigma^{24,3}\T \oplus \Sigma^{40,6}\T. $$
\end{prop}

\begin{proof}
Tensoring the splitting of Proposition~\ref{prop:split1} with $\bou_1$, we have 
$$ v_2^{-1} \bp{4} \simeq 2\Sigma^{16,1} v_2^{-1}\bp{2} \oplus \Sigma^{24,2}\T \otimes \bou_1.$$
Examination of $\piA_{*,*}(\bp{4})$ (Figure~\ref{fig:bo1^4}) reveals that
\begin{multline*}
 \piA_{*,*}(v_2^{-1} \bp{4}) \simeq 
 \\
 2\piA_{*,*}(\Sigma^{16,1} v_2^{-1}\bp{2}) \oplus \piA_{*,*}(\Sigma^{48,5}\T) \oplus \piA_{*,*}(\Sigma^{64,8}\T). 
 \end{multline*}
It follows that there is an isomorphism
$$ \piA_{*,*}(\T \otimes \bou_1) \cong \piA_{*,*}(\Sigma^{24,3}\T) \oplus \piA_{*,*}(\Sigma^{40,6}\T). $$
Moreover, one can check from the $\piA_{*,*}(\FF_2)$-module structure of $\piA_{*,*}(\bp{4})$ that the isomorphism preserves multiplication by 
$$ v_0, v_1^4, v_0v_2^2, v_2^8, h_1, h_2, g, v_2^4h_1. $$
The map
$$ \Sigma^{24,3}\FF_2 \oplus \Sigma^{40,6}\FF_2 \to \T \otimes \bou_1 $$
which maps the two generators in gives rise to a map of $\T$-modules
$$ \Sigma^{24,3}\T \oplus \Sigma^{40,6}\T \to \T \otimes \bou_1. $$
One can then use $\piA_{*,*}(\FF_2)$-module structures to determine that this map is an isomorphism on $\piA_{*,*}$.
\end{proof}

\begin{rmk}\label{rmk:bo1^ksplit}
Propositions~\ref{prop:split1} and \ref{prop:split2} allow one to inductively compute a splitting of $v_2^{-1}\bp{k}$ in $\mc{D}_{A(2)_*}$ as a sum of suspensions of $v_2^{-1}\bou_1$, $v_2^{-1}\bp{2}$ and $\T$.
For example, we have
\begin{align*}
v_2^{-1}\bp{4} & \simeq (2\Sigma^{16,1}v_2^{-1}\bou_1 \oplus \Sigma^{24,2}\T)\otimes \bou_1 \\
& 2\Sigma^{16,1}v_2^{-1}\bp{2} \oplus \Sigma^{24,2}\T\otimes \bou_1 \\
& 2\Sigma^{16,1}v_2^{-1}\bp{2} \oplus \Sigma^{48,5}\T \oplus \Sigma^{64,8}\T.
\end{align*}
In the next case, we can further simplify the answer using $v_2^8$ periodicity.
\begin{align*}
v_2^{-1}\bp{5} 
 \simeq  \: & (2\Sigma^{16,1}v_2^{-1}\bp{2} \oplus \Sigma^{48,5}\T \oplus \Sigma^{64,8}\T)\otimes \bou_1 \\
 \simeq \: &  2\Sigma^{16,1}v_2^{-1}\bp{3} \oplus \Sigma^{48,5}\T\otimes \bou_1 \oplus \Sigma^{64,8}\T\otimes \bou_1 \\
 \simeq \: &  4\Sigma^{32,2}v_2^{-1}\bou_1 \oplus 2\Sigma^{40,3}\T \oplus \Sigma^{72,8}\T \\
 & \: \oplus 2\Sigma^{88,11}\T \oplus \Sigma^{104,14}\T \\ 
 \simeq \: &  4\Sigma^{32,2}v_2^{-1}\bou_1 \oplus \Sigma^{24}\T \oplus 4\Sigma^{40,3}\T \oplus \Sigma^{56,6}\T.
\end{align*}
We similarly may compute
\begin{equation}\label{eq:bo1^6}
\begin{split}
v_2^{-1}\bp{6} \simeq \: & 4\Sigma^{32,2}v_2^{-1}\bp{2} \oplus \Sigma^{48,3}\T \oplus 5\Sigma^{64,6} \T \\
& \: \oplus 5\Sigma^{32,1}\T \oplus \Sigma^{48,4}\T.
\end{split}
\end{equation}
\end{rmk}

%

Finally, we will find the following splitting to be useful.

\begin{prop}\label{prop:tmf3^2splitting}
There is a splitting
$$ \T^{\otimes 2} \simeq \T \oplus \Sigma^{0,-1}\T \oplus \Sigma^{16,2} \T \oplus \Sigma^{32,5}\T. $$ 
\end{prop}

\begin{proof}
Smashing the splitting of Proposition~\ref{prop:split1} with itself, and applying Proposition~\ref{prop:split2} and $v_2^8$-periodicity, we have
\begin{align*}
 v_2^{-1}\bou_1^{\otimes 6} 
 & \simeq 4\Sigma^{32,2}\bp{2} \oplus 4\Sigma^{40,3}\bou_1 \otimes \T \oplus \Sigma^{48,4}\T^{\otimes 2} \\
  & \simeq 4\Sigma^{32,2}\bp{2} \oplus 4\Sigma^{64,6}\T \oplus 4\Sigma^{80,9}\T \oplus \Sigma^{48,4}\T^{\otimes 2} \\ 
  & \simeq 4\Sigma^{32,2}\bp{2} \oplus 4\Sigma^{64,6}\T \oplus 4\Sigma^{32,1}\T \oplus \Sigma^{48,4}\T^{\otimes 2}. \\ 
 \end{align*}
On the other hand, by (\ref{eq:bo1^6}), we have
\begin{align*} v_2^{-1}\bp{6} \simeq \: & 4\Sigma^{32,2}v_2^{-1}\bp{2} \oplus \Sigma^{48,3}\T \oplus 5\Sigma^{64,6} \T \\
& \: \oplus 5\Sigma^{32,1}\T \oplus \Sigma^{48,4}\T. 
\end{align*}
Making use of $\piA_{*,*}(\FF_2)$ module structures, we deduce that there is an isomorphism
\begin{multline*}
 \piA_{*,*}(\T^{\otimes 2}) \cong \\
 \piA_{*,*}(\Sigma^{0,-1}\T \oplus \Sigma^{16,2} \T \oplus \Sigma^{-16,-3}\T \oplus \T) \\
 \cong \piA_{*,*}(\Sigma^{0,-1}\T \oplus \Sigma^{16,2} \T \oplus \Sigma^{32,5}\T \oplus \T)
 \end{multline*}
 of $\piA_{*,*}(\FF_2)$-modules.  Since $\T^{\otimes 2}$ is a $\T$-module, we can extend the $\piA_{*,*}(\T)$-module generators of $\piA_{*,*}(\T^{\otimes 2})$ to a map
$$\Sigma^{0,-1}\T \oplus \Sigma^{16,2} \T \oplus \Sigma^{32,5}\T \oplus \T \to \T^{\otimes 2} $$
which is a $\piA_{*,*}$-isomorphism, hence an equivalence.
\end{proof}

\section{Generating functions}\label{sec:genfuncs}

In this section we will describe a useful combinatorial way of computing decompositions of $v_2^{-1}\bp{k}$ and $v_2^{-1}\bou_j$.

We will represent the objects of $\mc{D}_{A(2)_*}$ of the form
\begin{equation}\label{eq:genform}
\Sigma^{8i_1,j_1}v_2^{-1}\bp{k_1}\otimes \T^{\otimes l_1} \oplus \cdots \oplus \Sigma^{8i_n, j_n}v_2^{-1}\bp{k_n}\otimes \T^{\otimes l_n}
\end{equation}
by elements of $\ZZ[s^{\pm},t^{\pm},x,y]$:
$$
 t^{i_1}s^{j_1}x^{k_1}y^{l_1} + \cdots + t^{i_n}s^{j_n}x^{k_n}y^{l_n}. 
$$

Propositions~\ref{prop:split1}, \ref{prop:split2}, and $v_2$-periodicity impose some relations on this polynomial ring --- we therefore work in the quotient ring
\begin{equation}\label{eq:R}
 R := \ZZ[s^{\pm},t^{\pm},x,y]/(x^3 = 2t^2sx+t^3s^2y, \: xy := t^3s^3y+t^5s^6y, \: t^6s^8 = 1).
\end{equation}
Note that these relations imply
$$ y^2 = y + s^{-1}y + t^2s^2y + t^4s^5y. $$
This relation reflects the splitting of Prop~\ref{prop:tmf3^2splitting}.

We may use the relations of $R$ to reduce $x^k$ to a sum of monomials whose terms are of the form $t^is^jx$, $t^is^jx^2$, and $t^is^jy$.  These reduced forms of $x^k$ correspond to splittings of $v_2^{-1}\bp{k}$.  For example, the splitting (\ref{eq:bo1^6}) corresponds to the expression
$$ x^6 = 5 s^6 t^8 y + s^4 t^6 y + s^3 t^6 y + 5 s t^4 y + 4 s^2 t^4 x^2 $$
in $R$.  Table~\ref{tab:bo1^k} shows the reduced forms of $x^k$ in $R$ for $k \le 16$.

\begin{table}
\begin{tabular}{rcl}
\hline
&&\\
 $x^3$ & $=$ & $s^2 t^3 y + 2 s t^2 x$ \\
 $x^4$ & $=$ & $s^5 t^6 y + t^2 y + 2 s t^2 x^2$ \\
 $x^5$ & $=$ & $s^6 t^7 y + 4 s^3 t^5 y + t^3 y + 4 s^2 t^4 x$ \\
 $x^6$ & $=$ & $5 s^6 t^8 y + s^4 t^6 y + s^3 t^6 y + 5 s t^4 y + 4 s^2 t^4 x^2$ \\
 $x^7$ & $=$ & $6 s^7 t^9 y + s^6 t^9 y + 14 s^4 t^7 y + s^2 t^5 y + 6 s t^5 y + 8 s^3 t^6 x$ \\
 $x^8$ & $=$ & $20 s^7 t^{10} y + 7 s^5 t^8 y + 7 s^4 t^8 y + 20 s^2 t^6 y + s t^6 y + t^4 y + 8 s^3 t^6 x^2$ \\
 $x^9$ & $=$ & $8 s^7 t^{11} y + s^6 t^9 y + 48 s^5 t^9 y + s^4 t^9 y + 8 s^3 t^7 y + 27 s^2 t^7 y + 27 t^5 y$ \\
 && $+ 16 s^4 t^8 x$ \\
 $x^{10}$ & $=$ & $s^7 t^{12} y + 35 s^6 t^{10} y + 35 s^5 t^{10} y + s^4 t^8 y + 75 s^3 t^8 y + 9 s^2 t^8 y$ \\
 && $+ 9 s t^6 y + 75 t^6 y + 16 s^4 t^8 x^2$ \\
 $x^{11}$ & $=$ & $10 s^7 t^{11} y + 166 s^6 t^{11} y + 10 s^5 t^{11} y + 44 s^4 t^9 y + 110 s^3 t^9 y + s^2 t^9 y$ \\
 && $+ s^2 t^7 y + 110 s t^7 y + 44 t^7 y + 32 s^5 t^{10} x$ \\
 $x^{12}$ & $=$ & $154 s^7 t^{12} y + 154 s^6 t^{12} y + s^5 t^{12} y + 11 s^5 t^{10} y + 276 s^4 t^{10} y$ \\
 && $+ 54 s^3 t^{10} y + 54 s^2 t^8 y + 276 s t^8 y + 11 t^8 y + t^6 y + 32 s^5 t^{10} x^2$ \\
 $x^{13}$ & $=$ & $584 s^7 t^{13} y + 65 s^6 t^{13} y + s^6 t^{11} y + 208 s^5 t^{11} y + 430 s^4 t^{11} y$ \\
 && $+ 12 s^3 t^{11} y + 12 s^3 t^9 y + 430 s^2 t^9 y + 208 s t^9 y + t^9 y + 65 t^7 y + 64 s^6 t^{12} x$ \\
 $x^{14}$ & $=$ & $638 s^7 t^{14} y + 13 s^6 t^{14} y + 77 s^6 t^{12} y + 1014 s^5 t^{12} y + 273 s^4 t^{12} y$ \\
 && $ + s^3 t^{12} y + s^4 t^{10} y + 273 s^3 t^{10} y + 1014 s^2 t^{10} y + 77 s t^{10} y + 13 s t^8 y + 638 t^8 y$ \\
 && $ + 64 s^6 t^{12} x^2$ \\
 $x^{15}$ & $=$ & $350 s^7 t^{15} y + s^6 t^{15} y + 14 s^7 t^{13} y + 911 s^6 t^{13} y + 1652 s^5 t^{13} y$ \\
 && $ + 90 s^4 t^{13} y + 90 s^4 t^{11} y + 1652 s^3 t^{11} y + 911 s^2 t^{11} y + 14 s t^{11} y + s^2 t^9 y$ \\
 && $ + 350 s t^9 y + 2092 t^9 y + 128 s^7 t^{14} x$ \\
 $x^{16}$ & $=$ & $104 s^7 t^{16} y + 440 s^7 t^{14} y + 3744 s^6 t^{14} y + 1261 s^5 t^{14} y + 15 s^4 t^{14} y$ \\ 
 && $+ 15 s^5 t^{12} y + 1261 s^4 t^{12} y + 3744 s^3 t^{12} y + 440 s^2 t^{12} y + s t^{12} y + 104 s^2 t^{10} y$ \\
 && $ + 2563 s t^{10} y + 2563 t^{10} y + t^8 y + 128 s^7 t^{14} x^2$ \\
 &&\\
\hline
\end{tabular}
\caption{Reduced expressions for $x^k$ in $R$ corresponding to decompositions of $v_2^{-1}\bp{k}$.}\label{tab:bo1^k}
\end{table}

In light of Propositions~\ref{prop:Dbo1} we can also compute the duals of objects of the form (\ref{eq:genform}) represented as an element of $R$ via the ring map:
\begin{align*}
D: R & \to R \\
t & \mapsto t^{-1} \\
s & \mapsto s^{-1} \\
x & \mapsto t^{-2}s^{-1} \cdot x \\
y & \mapsto s \cdot y
\end{align*}
Note the formula $D(y) = sy$ is forced by the relations of $R$ since
\begin{align*}
2t^{-4}s^{-2}x + t^{-3}s^{-1}y & = t^{-6}s^{-3}x^3 \\
& = D(x^3) \\
& = D(2t^2sx+t^3s^2y) \\
& = 2t^{-4}s^{-2}x+t^{-3}s^{-2}Dy.
\end{align*}
We note however that Proposition~\ref{prop:split1} and Proposition~\ref{prop:Dbo1} can be used to deduce that $v_2^{-1}D\T \simeq \Sigma^{0,1}\T$.

\emph{Now assume that the connecting morphisms $\partial_j$ (\ref{eq:partialj}) are trivial for for $1 \le j \le j_0$.}  (We will eventually prove $\partial_j$ is always zero in Theorem~\ref{thm:partialj}.)
Then we can inductively define elements of $R$ which encode the splitting of $v_2^{-1}\bou_j$ for $j \le 2j_0+1$.  These are the \emph{bo-Brown-Gitler} polynomials, introduced in \cite[Sec.~8]{BHHM2}.  Their definition comes from (\ref{eq:v2boj2}) and (\ref{eq:splittingconj}).
\begin{equation}\label{eq:fj}
\begin{split}
 f_0 & := 1, \\ 
 f_1 & := x, \\
 f_{2j+1} & := t^{j} x \cdot f_{j}, \\
 f_{2j} & := t^{j}f_j + t^{j+1}s \cdot f_{j-1}.  
\end{split}
\end{equation} 
Table~\ref{tab:boj} shows reduced expressions for $f_j$ in $R$ for $j \le 16$.

\begin{table}
\begin{tabular}{rcl}
\hline
&& \\
 $f_1$ & $=$ & $x$ \\
 $f_2$ & $=$ & $t x + s t^2$ \\
 $f_3$ & $=$ & $t x^2$ \\
 $f_4$ & $=$ & $s t^3 x + t^3 x + s t^4$ \\
 $f_5$ & $=$ & $t^3 x^2 + s t^4 x$ \\
 $f_6$ & $=$ & $t^4 x^2 + s t^5 x + s^2 t^6$ \\
 $f_7$ & $=$ & $s^2 t^7 y + 2 s t^6 x$ \\
 $f_8$ & $=$ & $s t^6 x^2 + s t^7 x + t^7 x + s t^8$ \\
 $f_9$ & $=$ & $s t^7 x^2 + t^7 x^2 + s t^8 x$ \\
 $f_{10}$ & $=$ & $t^8 x^2 + s^2 t^9 x + 2 s t^9 x + s^2 t^{10}$ \\
 $f_{11}$ & $=$ & $s^2 t^{11} y + s t^9 x^2 + 2 s t^{10} x$ \\
 $f_{12}$ & $=$ & $s t^{10} x^2 + t^{10} x^2 + s^2 t^{11} x + s t^{11} x + s^2 t^{12}$ \\
 $f_{13}$ & $=$ & $s^2 t^{13} y + s t^{11} x^2 + s^2 t^{12} x + 2 s t^{12} x$ \\
 $f_{14}$ & $=$ & $s^2 t^{14} y + s t^{12} x^2 + s^2 t^{13} x + 2 s t^{13} x + s^3 t^{14}$ \\
 $f_{15}$ & $=$ & $s^5 t^17 y + t^{13} y + 2 s t^{13} x^2$ \\
 $f_{16}$ & $=$ & $s^3 t^{16} y + s t^{14} x^2 + 2 s^2 t^{15} x + s t^{15} x + t^{15} x + s t^{16}$ \\
 && \\
\hline
\end{tabular}
\caption{Reduced expressions for $f_j$ in $R$.}\label{tab:boj}
\end{table}

\section{$g$-local computations}\label{sec:glocal}

We will now consider the $g$-local bo-Brown-Gitler comodules, for 
$$ g = h^4_{2,1} \in \piA_{20,4}(\FF_2). $$  
The $g$-local results of this section will be crucial for the main result of Section~\ref{sec:partialj}.

Because the terms $A(2)\mmod A(1)_* \otimes \tmfu_{j-1}$ in (\ref{eq:boSES1}) and (\ref{eq:boSES2}) are $g$-locally acyclic in $\mc{D}_{A(2)_*}$, we have cofiber sequences
\begin{equation}\label{eq:glocal1}
 \Sigma^{8j}g^{-1}\bou_j \to g^{-1}\bou_{2j} \to \Sigma^{8j+8,1} g^{-1}\bou_{j-1} \xrightarrow{\partial'_j} \Sigma^{8j+1,-1} g^{-1}\bou_{j}
 \end{equation}
and equivalences
\begin{equation}\label{glocal2} 
g^{-1}\bou_{2j+1} \simeq \Sigma^{8j}g^{-1}\bou_j \otimes \bou_1.
\end{equation}

We therefore get a $g$-local story completely analogous to the $v_2$-local story, except much easier, because there are no `$\T$'-terms.

\begin{prop}\label{prop:split3}
There is a splitting
$$ g^{-1}\bou_1^{\otimes 3} \simeq 2\Sigma^{16,1} g^{-1}\bou_1. $$
\end{prop}

\begin{proof}
This follows the proof of Proposition~\ref{prop:split1}, except the situation is simpler because 
$$ g^{-1}(\bou_1^{\otimes 3})^{hC_3} \simeq 0 $$
since $g^{-1}\piA_{*,*}(\bou_1^{\otimes 3})^{C_3}$ is zero by inspection. 
\end{proof}

We also have the following $g$-local analog of Proposition~\ref{prop:Dbo1}, whose proof is identical.

\begin{prop}\label{prop:glocalDbo1}
We have
$$ g^{-1}D\bou_1 \simeq \Sigma^{-16,-1}g^{-1}\bou_1. $$
\end{prop}

Thus we may analyze the decompositions of $g^{-1}\bou_j$ by means of generating functions analogous to Section~\ref{sec:genfuncs}.  In light of Proposition~\ref{prop:split3}, instead of working in the ring $R$, we work in the ring 
$$ R' := \ZZ[s^{\pm},t^{\pm},x]/(x^3 = 2t^2sx). $$
By Proposition~\ref{prop:glocalDbo1}, we may encode $g$-local Spanier-Whitehead duality by the function
\begin{align*}
D: R' & \to R' \\
s & \mapsto s^{-1} \\
t & \mapsto t^{-1} \\
x & \mapsto t^{-2}s^{-1}x 
\end{align*}
Define elements $f'_j \in R'$ by the same inductive definition (\ref{eq:fj}) used to define the elements $f_j \in R$.
A simple induction reveals the following.

\begin{lem}\label{lem:f'j}
The elements $f'_j \in R'$ take the form
$$ f'_j = \begin{cases}
\sum_i (a_{i,j} s^i t^j + b_{i,j}s^i t^{j-1}x + c_{i,j}s^i t^{j-2} x^2), & j \: \rm{even}, \\
\sum_i (b_{i,j}s^i t^{j-1}x + c_{i,j}s^i t^{j-2} x^2), & j \: \rm{odd}, \\
\end{cases}
$$
for $a_{i,j}, b_{i,j}, c_{i,j} \in \mb{N}$. 
\end{lem}

\section{The attaching maps $\partial_j$ and $\partial_j'$}\label{sec:partialj}

\begin{thm}\label{thm:partialj}
The attaching maps $\partial_j$ (\ref{eq:partialj}) and $\partial'_j$ (\ref{eq:glocal1}) are zero for all $j$.
\end{thm}

\begin{proof}
Write the exact sequence (\ref{eq:boSES1}) as a splice of two short exact sequences
 $$ 
 \xymatrix@R-2em@C-1em{
 && 0 \ar[dr] && \quad 0 \quad \quad \quad \\
 &&& K \ar[dr] \ar[ur] \\
 0 \ar[r] & \Sigma^{8j} \ul{\bo}_j \ar[r] & \ul{\bo}_{2j} \ar[ur] \ar[rr] && A(2)\mmod A(1)_* \otimes \ul{\tmf}_{j-1}  \ar[r] & \Sigma^{8j+9} \ul{\bo}_{j-1} \ar[r] & 0} $$
and let
\begin{gather*}
\Sigma^{8j}\bou_j \to \bou_{2j} \to K \xrightarrow{\alpha} \Sigma^{8j+1,-1} \bou_{j} \\
\Sigma^{8j+8,1} \bou_{j-1} \xrightarrow{\beta} K \to A(2)\mmod A(1)_* \otimes \ul{\tmf}_{j-1}  \to \Sigma^{8j+9} \bou_{j-1}
\end{gather*}
be the cofiber sequences in $\mc{D}_{A(2)_*}$ induced from these short exact sequences.
Then we have the following commutative diagram in $\mc{D}_{A(2)_*}$.
\begin{equation*}
\xymatrix@C+1em{
\\
\Sigma^{8j+8,1}v_2^{-1}\bou_{j-1} \ar@/^{2pc}/[rr]^{\partial_j} \ar[r]^-{\simeq}_-{v_2^{-1}\beta} \ar[d] & v_2^{-1}K  \ar[r]_-{v_2^{-1}\alpha} \ar[d] &  \Sigma^{8j+1,-1} v_2^{-1}\bou_j  \ar[d] \\
\Sigma^{8j+8,1}v_2^{-1}g^{-1}\bou_{j-1} \ar[r]^-{\simeq}_-{v_2^{-1}g^{-1}\beta} & v_2^{-1}g^{-1}K  \ar[r]_-{v_2^{-1}g^{-1}\alpha} &  \Sigma^{8j+1,-1} v_2^{-1}g^{-1}\bou_j \\
\Sigma^{8j+8,1}g^{-1}\bou_{j-1} \ar@/_{2pc}/[rr]_{\partial_j'} \ar[u] \ar[r]^-{\simeq}_-{g^{-1}\beta} & g^{-1}K  \ar[u] \ar[r]_-{g^{-1}\alpha} &  \ar[u] \Sigma^{8j+1,-1} g^{-1}\bou_j }
\end{equation*}
We therefore have
\begin{equation}\label{eq:partialcompatible}
g^{-1}\partial_j = v_2^{-1} \partial_{j}'.
\end{equation}
Now, Assume inductively that $\partial_k$ and $\partial'_k$ are zero for $k < j$.  Then for $k < 2j+1$, $v_2^{-1}\bou_k$ and $g^{-1}\bou_k$ decomposes in $\mc{D}_{A(2)_*}$ as a sum of terms corresponding to the terms of $f_k$ and $f'_k$, respectively.  Note that we have
\begin{align*}
\partial_j \in \piA_{7,2}(v_2^{-1}D(\bou_{j-1})\otimes \bou_j), \\ 
\partial'_j \in \piA_{7,2}(g^{-1}D(\bou_{j-1})\otimes \bou_j). 
\end{align*}
It follows from Lemma~\ref{lem:f'j} that 
$$ D(f'_{j-1})\cdot f'_j = \sum_i(\alpha_i s^i x + \beta_i s^i t^{-1} x^2) $$
for $\alpha_i, \beta_i \in \mb{N}$, and therefore
\begin{equation}\label{eq:glocalDboj-1boj}
 g^{-1}D(\bou_{j-1})\otimes \bou_j \simeq \bigoplus_i(\alpha_i \Sigma^{0,i} g^{-1}\bou_1 + \beta_i \Sigma^{-8,i} g^{-1}\bp{2}).
 \end{equation}
Note that there is a map of rings 
$$ \phi: R' \to R $$ 
sending $s$ to $s$, $t$ to $t$, and $x$ to $x$.  
We have
$$ f_k \equiv \phi(f'_k) \mod y. $$
We therefore have
$$ D(f_{j-1})\cdot f_j = \sum_i(\alpha_i s^i x + \beta_i s^i t^{-1} x^2) + \sum_{k,l}\gamma_{k,l}s^kt^l y. $$
It follows that we have
\begin{equation}\label{eq:v2localDboj-1boj}
 v_2^{-1}D(\bou_{j-1})\otimes \bou_j \simeq \bigoplus_i(\alpha_i \Sigma^{0,i} v_2^{-1}\bou_1 + \beta_i \Sigma^{-8,i} v_2^{-1}\bp{2}) \oplus \bigoplus_{k,l}\gamma_{k,l}\Sigma^{8l,k}\T.
 \end{equation}

Note that 
$$ \piA_{8m+7,n}(\T) = 0 $$
for all $n,m$, so the the only potential non-zero components of $\partial_j$ under the decomposition (\ref{eq:v2localDboj-1boj}) are the components
\begin{align*}
(\partial_j)_i^{(1)} & \in \pi_{7,2-i}(\alpha_iv_2^{-1}\bou_1), \\
(\partial_j)_i^{(2)} & \in \pi_{15,2-i}(\beta_iv_2^{-1}\bp{2}).
\end{align*}
Similarly, let 
\begin{align*}
(\partial'_j)_i^{(1)} & \in \pi_{7,2-i}(\alpha_ig^{-1}\bou_1), \\
(\partial'_j)_i^{(2)} & \in \pi_{15,2-i}(\beta_ig^{-1}\bp{2})
\end{align*}
denote the components of $\partial'_j$ under the splitting (\ref{eq:glocalDboj-1boj}).  

Note that the splittings (\ref{eq:glocalDboj-1boj}) and (\ref{eq:v2localDboj-1boj}) are compatible under the maps
$$ g^{-1}D(\bou_{j-1})\otimes \bou_j \to v_2^{-1}g^{-1} D(\bou_{j-1})\otimes \bou_j \leftarrow v_2^{-1} D(\bou_{j-1})\otimes \bou_j $$
since $g^{-1}\T \simeq 0$, and by (\ref{eq:partialcompatible}) $\partial_j'$ and $\partial_j$ map to the same element of 
$$ \piA_{7,2}(v_2^{-1}g^{-1} D(\bou_{j-1})\otimes \bou_j). $$
We therefore deduce that under the maps
\begin{gather*}
\alpha_i g^{-1}\bou_1 \to \alpha_i v_2^{-1}g^{-1}\bou_1 \leftarrow \alpha_i v_2^{-1}\bou_1, \\
\beta_i g^{-1}\bp{2} \to \beta_i v_2^{-1}g^{-1}\bp{2} \leftarrow \beta_i v_2^{-1}\bp{2}
\end{gather*}
we have
\begin{align*}
 v_2^{-1} (\partial'_j)_i^{(1)} & = g^{-1}  (\partial_j)_i^{(1)}, \\
 v_2^{-1} (\partial'_j)_i^{(2)} & = g^{-1}  (\partial_j)_i^{(2)}. \\
\end{align*}
However, direct inspection of $\piA_{*,*}(\bou_1)$ and $\piA_{*,*}(\bp{2})$ reveals:
\begin{itemize}
\item The maps
\begin{gather*}
\piA_{7,s}(g^{-1}\bou_1) \hookrightarrow \piA_{7,s}(v_2^{-1}g^{-1}\bou_1) \hookleftarrow \piA_{7,s}(v_2^{-1}\bou_1), \\
\piA_{15,s}(g^{-1}\bp{2}) \hookrightarrow \piA_{15,s}(v_2^{-1}g^{-1}\bp{2}) \hookleftarrow \piA_{15,s}(v_2^{-1}\bp{2})
\end{gather*}
are injections for all $s$.
\vspace{10pt}

\item We have 
\begin{align*}
\piA_{7,s}(g^{-1}\bou_1) & = 0, \\
\piA_{15,s}(g^{-1}\bp{2}) & = 0
\end{align*}
for $s \ge 1$.
\vspace{10pt}

\item We have 
\begin{align*}
\piA_{7,s}(v_2^{-1}\bou_1) & = 0, \\
\piA_{15,s}(v_2^{-1}\bp{2}) & = 0
\end{align*}
for $s \le 1$.
\end{itemize}
It follows that we must have
\begin{align*}
(\partial_j)^{(1)}_i & = 0, \\
(\partial'_j)^{(1)}_i & = 0, \\
(\partial_j)^{(2)}_i & = 0, \\
(\partial'_j)^{(2)}_i & = 0.
\end{align*}
\end{proof}

\begin{cor}\label{cor:glocal}
We have 
$$ g^{-1}\bou_{2j} \simeq \Sigma^{8j} g^{-1}\bou_j \oplus \Sigma^{8j+8,1}g^{-1}\bou_{j-1}. $$
Therefore, if we write $f'_j$ in the form
$$ f'_j = \sum_i (a_{i,j} s^i t^j + b_{i,j}s^i t^{j-1}x + c_{i,j}s^i t^{j-2} x^2)
$$
then we have
\begin{multline*} 
g^{-1}\bou_j \simeq 
\bigoplus_i (a_{i,j} \Sigma^{8j,i}g^{-1}\FF_2 \oplus b_{i,j}\Sigma^{8(j-1),i} g^{-1}\bou_1 \oplus c_{i,j}\Sigma^{8(j-2),i}g^{-1}\bp{2}).
\end{multline*}
\end{cor}

\begin{cor}\label{cor:v2bo2j}
We have 
$$ v_2^{-1}\bou_{2j} \simeq \Sigma^{8j} v_2^{-1}\bou_j \oplus \Sigma^{8j+8,1}v_2^{-1}\bou_{j-1}. $$
Therefore, if we write $f_j$ in the form
$$ f_j = \sum_i (a_{i,j} s^i t^j + b_{i,j}s^i t^{j-1}x + c_{i,j}s^i t^{j-2} x^2)
+\sum_{k,l}d_{j,k,l}s^k t^l y $$
then we have
\begin{multline*} 
v_2^{-1}\bou_j \simeq 
\bigoplus_i (a_{i,j} \Sigma^{8j,i}v_2^{-1}\FF_2 \oplus b_{i,j}\Sigma^{8(j-1),i} v_2^{-1}\bou_1 \oplus c_{i,j}\Sigma^{8(j-2),i}v_2^{-1}\bp{2}) \\
\oplus \bigoplus_{k,l} d_{k,l}\Sigma^{8l,k}\T. \end{multline*}
\end{cor}

\begin{cor}\label{cor:main}
Consider the element 
$$ h := tf_1w + t^2f_2w^2 + t^3 f_3 w^3 \cdots \in R[[w]]. $$
Write the coefficient of $w^j$ in $h^n$ as
$$ \sum_i (a^{(n)}_{i,j} s^i t^{2j} + b^{(n)}_{i,j}s^i t^{2j-1}x + c^{(n)}_{i,j}s^i t^{2j-2} x^2)
+\sum_{j,k,l}d^{(n)}_{k,l}s^k t^l y $$
then the weight $8j$ summand of $v_2^{-1}\br{\tmfu}^{\otimes n}$ decomposes as 
\begin{multline*} 
\bigoplus_i (a^{(n)}_{i,j} \Sigma^{16j,i}v_2^{-1}\FF_2 \oplus b^{(n)}_{i,j}\Sigma^{16j-8,i} v_2^{-1}\bou_1 \oplus c^{(n)}_{i,j}\Sigma^{16j-16,i}v_2^{-1}\bp{2}) \\
\oplus \bigoplus_{k,l} d^{(n)}_{j,k,l}\Sigma^{8l,k}\T. \end{multline*}
\end{cor}

\section{Applications to the $g$-local algebraic $\tmf$-resolution}\label{sec:BBT}

Consider the quotient Hopf algebra $C_* := \FF_2[\zeta_2]/(\zeta_2^4)$ of $A(2)_*$, with
$$ \pi^{C_*}_{*,*}(\FF_2) = \FF_2[v_1, h_{2,1}]. $$
The second author, Bobkova, and Thomas computed the $P_2^1$-Margolis homology of the tmf-resolution, and in the process computed the structure of $A \mmod A(2)^{\otimes n}_*$ as $C_*$-comodules.  From this one can read off the Ext groups
$$ h^{-1}_{2,1}\pi^{C_*}_{*,*}(\tmfu^{\otimes n}) $$
(see \cite[Thm.~3.12]{tmfhi}).

The groups $h^{-1}_{2,1}\pi^{C_*}_{*,*}$ are closely related to the groups $g^{-1}\piA_{*,*}$.  In \cite[Cor.~3.11]{tmfhi}, it is proven that for $M \in \mc{D}_{A(2)_*}$, there is a $v_2^8$ Bockstein spectral sequence
\begin{equation}\label{eq:P21ss}
h^{-1}_{2,1}\pi_{*,*}^{C_*}(M)\otimes \FF_2[v_2^8] \Rightarrow g^{-1}\piA_{*,*}(M).
\end{equation} 
In this section we would like to explain how Corollary~\ref{cor:glocal} can be used to compute $g^{-1}\piA_{*,*}(\tmfu^{\otimes n})$.  By relating this to \cite{BBT}, we will show that in the case of $M = \tmfu^{\otimes n}$, the spectral sequence (\ref{eq:P21ss}) collapses (Theorem~\ref{thm:glocal}).

We follow \cite{tmfhi} in our summary of the results of \cite{BBT}.
The coaction of $C_*$ is encoded in the dual action of the algebra $E[Q_1, P_2^1]$ on $\tmfu^{\otimes n}$.
Define elements
\begin{align*}
x_{i,j} & = 1\otimes  \cdots \otimes 1 \otimes \underbrace{\zeta_{i+3}}_{j} \otimes 1 \otimes \cdots \otimes 1,  \\
t_{i,j} & = 1\otimes \cdots \otimes 1 \otimes \underbrace{\zeta^{4}_{i+1}}_{j} \otimes 1 \otimes \cdots \otimes 1
\end{align*}
in $\tmfu^{\otimes n}$.  

For an \emph{ordered} set 
$$ J = ((i_1, j_1), \ldots, (i_k, j_k)) $$  
of multi-indices, let
$$ \abs{J} := k $$
denote the number of pairs of indices it contains.
Define linearly independent sets of elements
$$ \mc{T}_J \subset \tmfu^{\otimes n} $$
inductively as follows.  Define
$$ \mc{T}_{(i,j)} = \{x_{i,j}\}. $$
For $J$ as above with $\abs{J}$ odd, define
\begin{align*}
\mc{T}_{J, (i,j)} & = \{ z \cdot x_{i,j} \}_{z \in \mc{T}_J}, \\
\mc{T}_{J, (i,j), (i',j')} & = \{ Q_1(z \cdot x_{i,j})x_{i',j'} \}_{z \in \mc{T}_J} \cup \{ Q_1(z \cdot x_{i',j'})x_{i,j} \}_{z \in \mc{T}_J}.
\end{align*} 
Let 
$$ N_J \subset \tmfu^{\otimes n} $$ 
denote the $\FF_2$-subspace with basis
$$ Q_1 \mc{T}_J := \{ Q_1(z) \}_{z \in \mc{T}_J}. $$
While the set $\mc{T}_J$ depends on the ordering of $J$, the subspace $N_J$ does not.

Finally, for a set of pairs of indices
$$ J = \{(i_1,j_1), \cdots , (i_k,j_k)\} $$ 
as before, define
$$ x_Jt_J := x_{i_1,j_1}t_{i_1,j_1} \cdot \cdots \cdot x_{i_k,j_k}t_{i_k,j_k}. $$

The following can be read off of the computations of \cite{BBT}.

\begin{thm}[Bhattacharya-Bobkova-Thomas]\label{thm:BBT}
As modules over $\FF_2[h^\pm_{2,1}, v_1]$, we have
\begin{multline*} 
h_{2,1}^{-1}\pi^{C_*}_{*,*}(\tmfu_*^{\otimes n}) = \\
\FF_2[h_{2,1}^{\pm}] \otimes \bigg( \FF_2[v_1]\{x_{J'}t_{J'}\}_{J'} \oplus 
\bigoplus_{\abs{J} \: \mr{odd}} N_J\{x_{J'}t_{J'}\}_{J \cap J' = \emptyset} 
\\
\oplus \bigoplus_{\abs{J} \ne 0 \: \mr{even}}
\FF_2[v_1]/v_1^2 \otimes N_J\{x_{J'}t_{J'}\}_{J \cap J' = \emptyset} \bigg)
\end{multline*}
where $J$ and $J'$ range over the subsets of
$$ \{ (i,j) \: : \: 1 \le i, 1 \le j \le n \} $$ 
and $v_1$ acts trivially on $N_J$ for $\abs{J}$ odd.
\end{thm}

We now explain how the equivalences
\begin{align*}
 g^{-1} \bou_{2j} & \simeq  \Sigma^{8j} g^{-1}\bou_j \oplus \Sigma^{8j+8,1}g^{-1}\bou_{j-1}, \\
 g^{-1} \bou_{2j+1} & \simeq \Sigma^{8j} g^{-1} \bou_j \otimes \bou_1 
\end{align*}
are related to Theorem~\ref{thm:BBT}.  This analysis comes from the definitions of the maps in the exact sequences (\ref{eq:boSES1}) and (\ref{eq:boSES2}).  The definitions of these maps are give in \cite[Sec.~7]{BHHM}.
For a set $J$ of indices of the form
$$ J = \{ (i_1,1), \cdots, (i_k,1) \}, $$
define $J+\Delta$ to be the set
$$ J+\Delta = \{ (i_1+1,1), \cdots, (i_k+1,1) \}. $$
Then the induced maps on homotopy are determined by:
\begin{align*}
\piA_{*,*}(\Sigma^{8j} g^{-1}\bou_j) & \to \piA_{*,*}(g^{-1}\bou_{2j}) \\
 N_{J}\{x_{J'}t_{J'}\} & \mapsto N_{J+\Delta}\{x_{J'+\Delta}t_{J'+\Delta}\} \\
 \\
 \piA_{*,*}(\Sigma^{8j+8,1} g^{-1}\bou_{j-1} & 
 \to \piA_{*,*}(g^{-1}\bou_{2j}) \\
 N_{J}\{x_{J'}t_{J'}\} & \mapsto h_{2,1} \cdot N_{J+\Delta}\{x_{1,1}t_{1,1}x_{J'+\Delta}t_{J'+\Delta}\} \\
 \\
 \piA_{*,*}(\Sigma^{8j} g^{-1}\bou_j \otimes \bou_1) & = \piA_{*,*}(g^{-1}\bou_{2j+1}) \\
 N_{J \cup \{(1,2)\}}\{x_{J'}t_{J'}\} & \mapsto N_{(J+\Delta) \cup \{(1,1)\}}\{x_{J'+\Delta}t_{J'+\Delta}\}.
\end{align*}
We have (with $g = h^4_{2,1}$)
\begin{align*}
\piA_{*,*}(g^{-1}\FF_2) & = \FF_2[h^{\pm}_{2,1}, v_1, v_2^8], \\
\piA_{*,*}(g^{-1}\bou_1) & = \FF_2[h^{\pm}_{2,1}, v_1, v_2^8]/(v_1) \{t_{1,1}\}, \\
\piA_{*,*}(g^{-1}\bp{2}) & = \FF_2[h^{\pm}_{2,1}, v_1, v_2^8]/(v_1^2) \{Q_1(x_{1,1}x_{1,2})\}.
\end{align*}

Corollary~\ref{cor:glocal} therefore implies the following extension of Theorem~\ref{thm:BBT}.

\begin{thm}\label{thm:glocal}
As modules over $\FF_2[h^\pm_{2,1}, v_1, v_2^8]$, we have
\begin{multline*} 
g^{-1}\piA_{*,*}(\tmfu_*^{\otimes n}) = \\
\FF_2[h_{2,1}^{\pm}, v_2^8] \otimes \bigg( \FF_2[v_1]\{x_{J'}t_{J'}\}_{J'} \oplus 
\bigoplus_{\abs{J} \: \mr{odd}} N_J\{x_{J'}t_{J'}\}_{J \cap J' = \emptyset} 
\\
\oplus \bigoplus_{\abs{J} \ne 0 \: \mr{even}}
\FF_2[v_1]/v_1^2 \otimes N_J\{x_{J'}t_{J'}\}_{J \cap J' = \emptyset} \bigg)
\end{multline*}
where $J$ and $J'$ range over the subsets of
$$ \{ (i,j) \: : \: 1 \le i, 1 \le j \le n \} $$ 
and $v_1$ acts trivially on $N_J$ for $\abs{J}$ odd.
\end{thm}

\appendix

\section{Charts for $\pi^{A(2)_*}_{*,*} (\bou_1^{\otimes k})$ for $0
\le k \le 4$ and $\pi^{A(2)_*}_{*,*}(\T)$.}\label{apx:charts}

This appendix contains the charts for the homotopy groups of the various fundamental components of the $v_2$-local algebraic $\tmf$-resolution.

\begin{figure}[ph!]
\includegraphics[angle = 90, origin=c, height =.7\textheight]{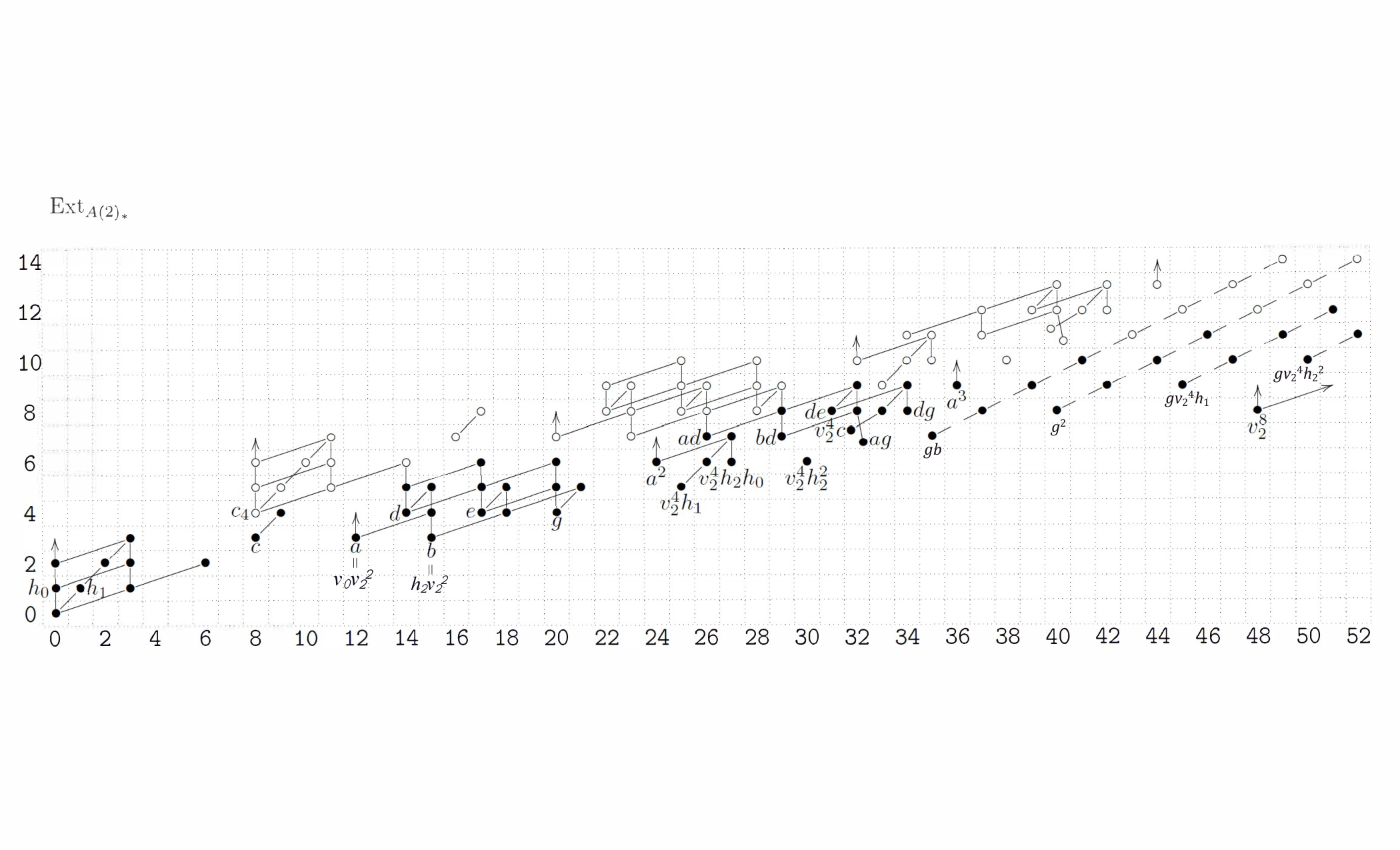}
\caption{$\pi^{A(2)_*}_{*,*}(\FF_2)$.}\label{fig:ExtA2}
\end{figure}

\afterpage{
\clearpage
\begin{figure}
\includegraphics[angle = 90, origin=c, height =.7\textheight]{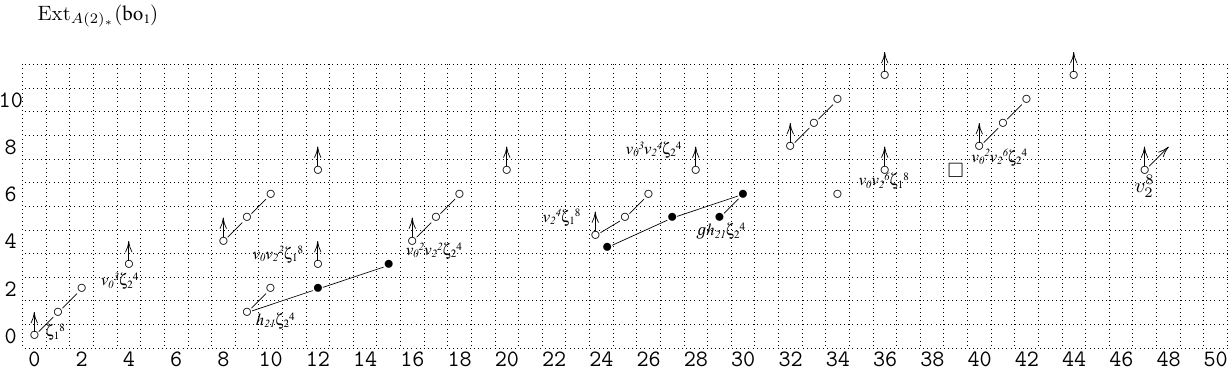}
\caption{$\pi^{A(2)_*}_{*,*}(\bou_1)$.}\label{fig:bo1}
\end{figure}
}

\afterpage{
\clearpage
\begin{figure}
\includegraphics[angle = 90, origin=c, height =.7\textheight]{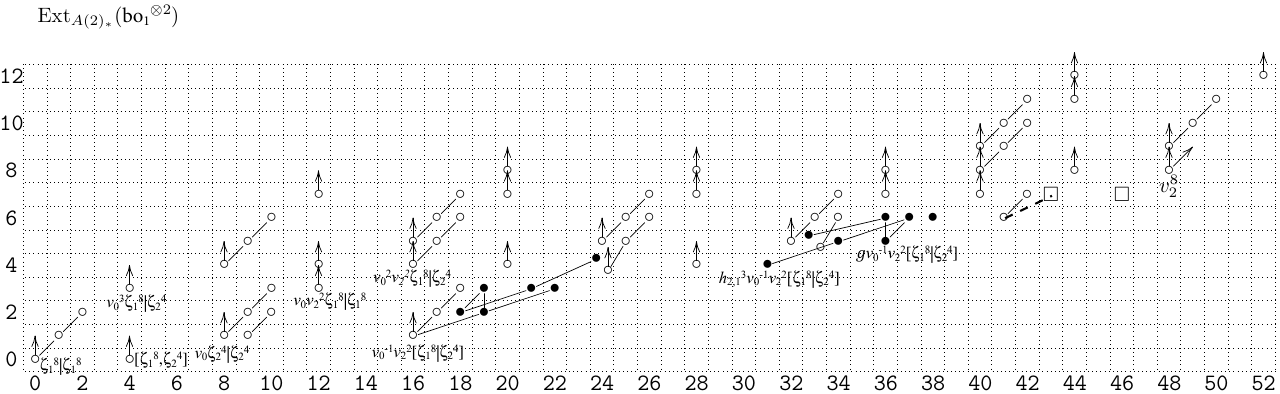}
\caption{$\pi_{*,*}^{A(2)_*}(\bou_1^{\otimes 2})$.}\label{fig:bo1^2}
\end{figure}
}

\afterpage{
\clearpage
\begin{figure}
\includegraphics[angle = 90, origin=c, height =.7\textheight]{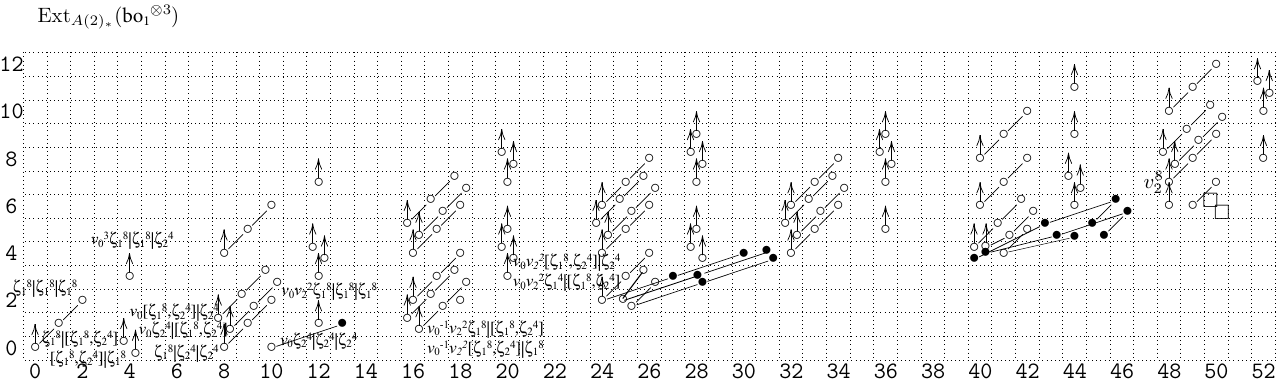}
\caption{$\pi_{*,*}^{A(2)_*}(\bou_1^{\otimes 3})$.}\label{fig:bo1^3}
\end{figure}
}

\afterpage{
\clearpage
\begin{figure}
\includegraphics[angle = 90, origin=c, height =.6\textheight]{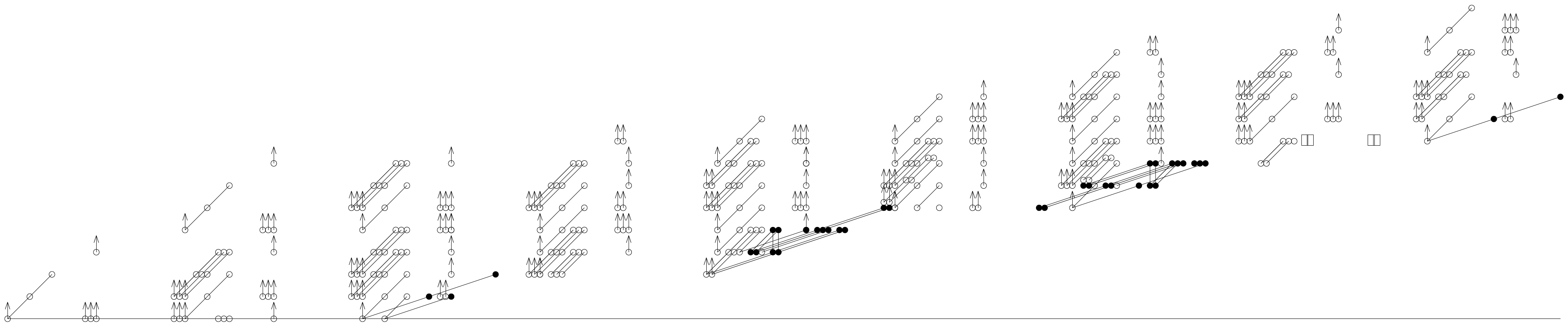}
\caption{$\pi_{*,*}^{A(2)_*}(\bou_1^{\otimes 4})$.}\label{fig:bo1^4}
\end{figure}
}

\afterpage{
\clearpage
\begin{figure}
\includegraphics[angle = 90, origin=c, height =.6\textheight]{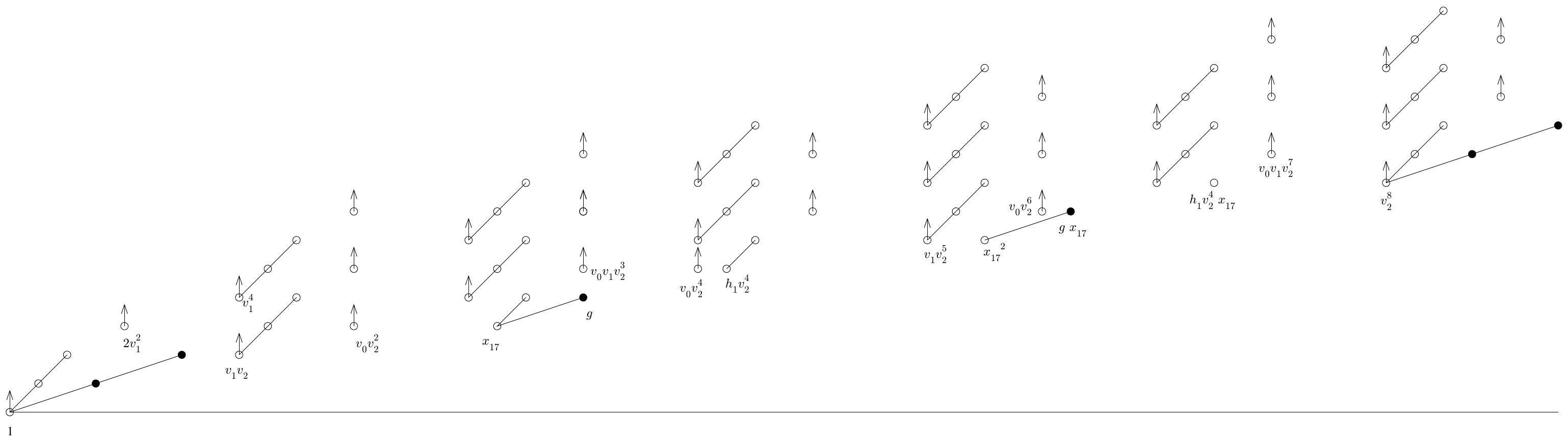}
\caption{$\pi_{*,*}^{A(2)_*}(\T)$.}\label{fig:TMF3}
\end{figure}
}

\afterpage{\clearpage}

\section{A splitting of $\bo_1^{\s 3}$ } \label{appendix}
The $v_2$-local splitting of Proposition~\ref{prop:split1} comes from a stable splitting of $\bo_1^{\s 3}$ induced by an idempotent decomposition of the identity element 
 \[ 1 = {\sf f}_1 + {\sf f}_2 + {\sf e} \in \mathbb{Z}_{(2)}[\Sigma_3]\] 
 as described in Remark~\ref{idempotent}.
\begin{figure}[h] 
\begin{small}
\begin{verbatim}
20

0 2 3 4 6 6 7 7 8 9 9 10 10 11 12 13 13 14 15 16

\end{verbatim}
\end{small} 
\begin{subfigure}{.33\textwidth}
\begin{small}
\begin{verbatim}
0 2 1 1
0 3 1 2 
0 4 1 3
0 6 1 4
0 7 1 6

1 1 1 2
1 4 1 5
1 5 1 7
1 6 1 8
1 7 1 9

2 4 1 7 
2 6 1 10
2 7 1 12

3 2 1 4
3 3 1 6
3 4 1 8
3 5 1 9
3 6 1 12

4 1 1 6
4 4 1 11
4 5 1 13
4 6 1 14

\end{verbatim}
\end{small}
\end{subfigure}
\begin{subfigure}{.33\textwidth}
\begin{small}
\begin{verbatim}

4 7 1 15

5 1 1 7
5 2 1 8
5 3 1 9
5 4 1 12

6 2 1 9
6 4 1 13
6 6 1 16
6 7 1 17

7 2 1 10
7 3 1 12

8 1 1 9
8 2 1 12
8 4 1 14
8 5 1 15
8 6 1 17

9 4 1 15
9 6 1 18
9 7 1 19

10 1 1 12

\end{verbatim}
\end{small}
\end{subfigure}
\begin{subfigure}{.32\textwidth}
\begin{small}
\begin{verbatim}

10 4 1 16
10 5 1 17

11 1 1 13
11 2 1 14
11 3 1 15
11 4 1 17

12 4 1 17
12 6 1 19

13 2 1 16
13 3 1 17
13 4 1 18
13 5 1 19

14 1 1 15
14 2 1 17

15 2 1 18
15 3 1 19

16 1 1 17

17 2 1 19
18 1 1 19


\end{verbatim}
\end{small}
\end{subfigure}
\caption{The $A(2)$-module structure of $H^*(F_1) \cong H^*(F_2)$ as an input file for Bruner's program} \label{BrunerF}
\end{figure}
 More precisely, if we set  
 \[ F_i := \hocolim \{ \bo_1^{\s 3 } \overset{{\sf f}_i}{\longrightarrow} \bo_1^{\s 3 } \overset{{\sf f}_i}{\longrightarrow} \dots  \}
 \]  for $i \in \{1, 2\}$  and \[ E :=  \hocolim \{ \bo_1^{\s 3 } \overset{{\sf e}}{\longrightarrow} \bo_1^{\s 3 } \overset{{\sf e}}{\longrightarrow} \dots  \}
,\] 
using the evident permutation action of $\Sigma_3$ on $\bo_1^{\s 3}$, then it is easy to see that 
\begin{equation} \label{splitbo1^3}
\bo_1^{\s 3} \simeq F_1 \vee F_2 \vee E. 
\end{equation}
 In fact, $F_1$, $F_2$ and $E$ are finite spectra and their mod $2$ cohomology as a Steenrod module can be easily computed  using  the cocommutativity of Steenrod operations and a K\"unneth isomorphism (see \cite[Appendix C]{Orange}).  For the purposes of this paper, we only need their underlying $A(2)$-module structure which we record in the format of a Bruner module definition file \cite[Apx.~A]{BEM} (see Figure~\ref{BrunerF} and Figure~\ref{BrunerE})
  \begin{rmk} \label{idempotent} In the group ring  $\mathbb{Z}_{(2)}[\Sigma_3]$, the identity element $1$ can be written as a sum of idempotent elements 
\[   {\sf f}_1 = \frac{1 + (1 \ 2) - (1\ 3) - (1 \ 2 \ 3)}{3},{\sf f}_2 = \frac{1 + (1 \ 3) - (1\ 2) - (1 \ 3 \ 2)}{3} \text{ and }  \]
\[ {\sf e} =\frac{ 1 + (1\ 2 \ 3) + (1\ 3 \ 2)}{3}  .\]
\end{rmk}
 \begin{figure}[h] 
 \begin{small}
\begin{verbatim}
24

0 4 6 7 8 10 10 11 11 12 12 13 13 14 14 15 16 17 17 18 18 19 20 21 

\end{verbatim}
\end{small}

 \begin{subfigure}{.33\textwidth}
\begin{small}
\begin{verbatim}

0 4 1 1 
0 6 1 2
0 7 1 3

1 2 1 2
1 3 1 3

2 1 1 3
2 4 2 5 6
2 5 2 7 8 

3 4 2 7 8
3 6 2 11 12

4 2 2 5 6
4 3 2 7 8
4 4 2 9 10
4 5 2 11 12
4 6 2 13 14
4 7 1 15

5 1 1 7
5 2 1 10
5 3 2 11 12
5 4 2 13 14
5 5 1 15

6 1 1 8
6 2 1 10
6 3 2 11 12 
6 4 2 13 14
6 5 1 15

7 2 1 11
7 3 1 14
7 4 1 15
\end{verbatim}
\end{small}
\end{subfigure}
\begin{subfigure}{.32 \textwidth}
\begin{small}
\begin{verbatim}
7 6 2 17 18

8 2 1 12
8 3 1 14
8 4 1 15
8 6 2 17 18

9 2 1 13
9 3 1 15
9 4 1 16
9 5 2 17 18
9 6 2 19 20
9 7 1 21

10 1 2 11 12
10 2 1 14
10 4 1 16
10 5 2 17 18
10 6 2 19 20
10 7 1 21

11 1 1 14
11 4 1 17
11 5 1 20
11 6 1 21

12 1 1 14
12 4 1 18
12 5 1 20
12 6 1 21


13 1 1 15
13 4 1 19
13 5 1 21
\end{verbatim}
\end{small}
\end{subfigure}
\begin{subfigure}{.33 \textwidth}
\begin{small}
\begin{verbatim}
13 6 1 22
13 7 1 23

14 4 1 20
14 6 1 22
14 7 1 23

15 2 2 17 18
15 4 1 21
15 6 1 23

16 1 2 17 18
16 2 2 19 20
16 3 1 21
16 4 1 22
16 5 1 23

17 1 1 20
17 2 1 21
17 4 1 23

18 1 1 20
18 2 1 21
18 4 1 23

19 1 1 21
19 2 1 22 
19 3 1 23

20 2 1 22
20 3 1 23

21 2 1 23

22 1 1 23
\end{verbatim}
\end{small}
\end{subfigure}
\caption{The $A(2)$-module structure of $H^*(E)$ as an input file for Bruner's program} \label{BrunerE}
\end{figure}
\begin{rmk} \label{F1=F2} Note that ${\sf f}_1$ and ${\sf f}_2$ are conjugates and therefore, $F_1\simeq F_2$.  
\end{rmk} 
 Bruner's program is capable of computing the action of $\pi_{*,*}^{A(2)_*} (\mathbb{F}_2)$ on $\pi_{*,*}^{A(2)_*} (M^{\vee})$, where $M^{\vee}$ is the $\FF_2$-linear dual of a finite $A(2)$-module $M$. Therefore, it can be used for verifying the details necessary in the proof of Proposition~\ref{prop:split1} and Proposition~\ref{prop:split2}.  
  \begin{rmk} \label{R2}
Using Bruner's program and  Figure~\ref{fig:TMF3}  one can easily verify 
\[ 
 v_2^{-1} \pi_{*,*}^{A(2)_*}(H_*(E)) \cong \pi_{*,*}^{A(2)_*}(\Sigma^{24,2} \T).
\]
Then by Theorem \ref{thm:recprinc} we get 
$\Sigma^{24,2}\T \simeq v_2^{-1} H_*(E)$ in 
$\mc{D}_{A(2)_*}$. 
\end{rmk}

\begin{rmk}[A different proof of Proposition~\ref{prop:split1}] 
Let $M_1$ denote the first integral Brown-Gitler module. It consists of three $\mathbb{F}_2$-generators $\{ x_0, x_2, x_3 \}$ where $|x_i| =i$ such that 
\[ \text{ $\mathit{Sq}^2(x_0)= x_2$ and $\mathit{Sq}^1(x_2)= x_3$.} \]
  It is tedious but straightforward to check that there is a short exact sequence 
\[ 
0 \to H^*(\Sigma^{17}\bo_1)\longrightarrow \Sigma^4 A(2) \mmod A(1) \otimes  M_1 \longrightarrow H^*E \to 0
\]
of $A(2)$-modules. This short exact sequence translates into an $\mc{D}_{A(2)_*}$-equivalence 
\[ v_2^{-1}H_*(F_1) \cong H_*(F_2)  \simeq  \Sigma^{16,1} v_2^{-1}\bou_1 \]
  which, along with Remark~\ref{R2} and  \eqref{splitbo1^3}, gives yet another proof of Proposition~\ref{prop:split1}. 
\end{rmk}

\bibliographystyle{amsalpha}
\bibliography{algTMFres}
\end{document}